\documentclass[12pt,reqno]{amsart}
\usepackage[a4paper]{geometry}
\usepackage[utf8]{inputenc}
\usepackage{versions}
\usepackage{stmaryrd}
\usepackage{xspace}
\usepackage{color}
\usepackage{dsfont}

\allowdisplaybreaks

\numberwithin{equation}{section}

\newtheorem{theo}{Theorem}[section]
\newtheorem{lemme}[theo]{Lemma}
\newtheorem{prop}[theo]{Proposition}
\newtheorem{corol}[theo]{Corollary}

\newtheorem{hypothesis}[theo]{Assumption}

\newcommand{\ind}{\ensuremath{\mathds{1}}\xspace}

\newcommand{\E}{\ensuremath{\mathbb{E}}\xspace}

\renewcommand{\H}{\ensuremath{\mathbb{H}}\xspace}

\newcommand{\N}{\ensuremath{\mathbb{N}}\xspace}

\newcommand{\R}{\ensuremath{\mathbb{R}}\xspace}
\renewcommand{\S}{\ensuremath{\mathbb{S}}\xspace}

\newcommand{\BB}{\ensuremath{\mathcal{B}}\xspace}
\newcommand{\CC}{\ensuremath{\mathcal{C}}\xspace}
\newcommand{\DD}{\ensuremath{\mathcal{D}}\xspace}

\newcommand{\HH}{\ensuremath{\mathcal{H}}\xspace}

\newcommand{\LL}{\ensuremath{\mathcal{L}}\xspace}
\newcommand{\MM}{\ensuremath{\mathcal{M}}\xspace}

\DeclareMathOperator{\GL}{GL}

\def\Var{{\mathrm{{\rm Var}}}}

\def\Cov{{\mathrm{{\rm Cov}}}}

\def\and{{\mathrm{{\rm and}}}}

\newcommand{\BP}{\ensuremath{\mathbf{P}}\xspace}
\newcommand{\BQ}{\ensuremath{\mathbf{Q}}\xspace}

\def\mathpal#1{\mathop{\mathchoice{\text{\rm #1}}%
   {\text{\rm #1}}{\text{\rm #1}}%
   {\text{\rm #1}}}\nolimits}

\def\End{\mathpal{End}}

\def\Ric{\mathpal{Ric}}

\def\id{\mathpal{id}}

\def\vol{\mathpal{vol}}

\begin{document}

\title[]{Intertwining relations for diffusions in manifolds and applications to functional inequalities}
%\date{\today\ \emph{ File: }\jobname.tex}

\author[B. Huguet]{Baptiste Huguet} 
\address{Institut de Mathématiques de Bordeaux, UMR CNRS 5251, Université de Bordeaux, France}
\email{baptiste.huguet@math.u-bordeaux.fr}
\urladdr{https://www.math.u-bordeaux.fr/~bhuguet/}

%%%%%%%%%%%%%%%%%%%%%%%%%%%%%%%%%%%%%%%%%%%%%%%%%%%%%%%%%%%%%%%%%%%%%%%%%%
%
%  Abstract, Keywords, AMS classification
%
%%%%%%%%%%%%%%%%%%%%%%%%%%%%%%%%%%%%%%%%%%%%%%%%%%%%%%%%%%%%%%%%%%%%%%%%%%

\begin{abstract}\noindent
We construct a generalisation of Bakry-\'Emery criterion to prove twisted intertwining relations for Markov semigroups. These relations are applied to Brascamp-Lieb type inequalities and spectral gap results. It extends the method of \cite{ABJ} to Riemannian manifolds and to a wider class of twist. These results are illustrated with several examples. 

\end{abstract}

\maketitle
\tableofcontents

%%%%%%%%%%%%%%%%%%%%%%%%%%%%%%%%%%%%%%%%%%%%%%%%%%%%%%%%%%%%%%%%%%%%%%%%%%%%
%
%  Actual Body of the Paper
%
%%%%%%%%%%%%%%%%%%%%%%%%%%%%%%%%%%%%%%%%%%%%%%%%%%%%%%%%%%%%%%%%%%%%%%%%%%%%
%%%%%%%%%%%%%%%%%%%%%%%%%%%%%%%%%%%%%%%%%%%%%%%%%%%%%%%%%%%%%%%%%%%%%%%%%%%%
\section{Introduction}\label{Section1}
\setcounter{equation}0
%%%%%%%%%%%%%%%%%%%%%%%%%%%%%%%%%%%%%%%%%%%%%%%%%%%%%%%%%%%%%%%%%%%%%%%%%%%%
The aim of this paper is to extend our understanding of intertwining relations between Markov semigroups in the setting of Riemannian manifolds and its applications in functional inequalities but also the underlying role of stochastic processes as the deformed parallel translation. These relations have been first investigated in the discrete case for birth-death processes in \cite{ChJ} and in the one dimensional case in~\cite{BJ}. The case of reversible ergodic diffusions in the Euclidean space $\R^n$ is treated in ~\cite{ABJ}. In this paper, we also investigate the case of reversible and ergodic diffusions, with generator 
\begin{align*}
Lf = \Delta f-\langle\nabla V,\nabla f\rangle
\end{align*}
where $V$ is a smooth potential on a Riemannian manifold $M$. Such a diffusion admits a unique invariant measure, $\mu$, absolutely continuous with respect to the Riemannian measure, with density proportional to $e^{-V}$. 

We are looking for intertwining relations by differentiation : the goal is to rewrite the derivative of a smooth Markov semigroup acting on functions as a Markov semigroup acting on differential forms. Unlike in the one-dimensional case, where functions and their derivatives have the same nature, in a manifold setting, the two intertwined semigroups act on different spaces. Actually, we look at semigroups on $1$-forms which restrictions on differential forms satisfy an intertwining relation. As we want to stress on the action on $1$-forms, we do not look into Bismut type formulae (see \cite{EL1} or \cite{EL2}).

At the level of operators, the intertwining relation occurs without further assumptions. The generator $L$ is intertwined with a weighted Laplacian acting on~$1$-forms, $L^W$, unitary equivalent to the Witten Laplacian. A large study of this operator can be found in the work of Helffer, with application to correlation decay in spin systems (see \cite{Hel}). At the level of stochastic processes, $L^W$ is the generator on $1$-forms of a diffusion on the tangent bundle: the deformed parallel translation (or geodesic transport in \cite{Mey}). In \cite{ACT}, this process appears naturally as a spacial derivative of a flow of the diffusion with generator $L$. These intertwining relations at the level of processes and generator suggest an intertwining relation at the level of the semigroups and a stochastic representation of the intertwined semigroup. However, at the level of semigroups, intertwining relations are not so obvious: more assumptions are required. In the Euclidean space, the classical assumption is the strong convexity of the potential~$V$ or, in other way to say it, the positiveness of its Hessian. A classical generalisation of this condition on Riemannian manifolds is the positiveness of an operator depending on the Hessian and the Ricci curvature, known as the Bakry-\'Emery criterion (see \cite{Bak}). The stochastic approach of intertwining relation is an important part of Li's PhD thesis \cite{LiXM1} and works (\cite{LiXM2}, \cite{LiXM3}) and the books of Elworthy, Le Jan and Li \cite{ELL} and \cite{ELL1}. In this paper, we extend a strategy presented in~\cite{ABJ} on~$\R^n$, to manifolds, so as to obtain intertwined semigroups even if the Bakry-\'Emery criterion is not fulfilled. We consider twisted gradients, or equivalently, twisted metrics on the tangent space by a section of~$\GL(TM)$. This operation does not change the stochastic diffusion on $M$ but creates new ones on the tangent space, with associated generators and semigroups. Under assumptions on these twists, which replace the Bakry-Émery condition, we can obtain intertwining relations at the level of semigroups. A consequence of these intertwinings is a family of Brascamp-Lieb type inequalities, extending the classical case satisfied under the strict convexity assumption of the potential. 

Let us summarize the content of this paper. In Section \ref{Section2}, we recall basic facts about semigroup, deformed parallel translation and the classical commutation at the level of the generators. The semigroup considered is stochastically defined on bounded continuous functions. In Section \ref{Section3}, we gives a stochastic proof of intertwining relation for the $\CC^0$ semigroups under the classical Bakry-\'Emery condition. It is applied to an asymmetric Brascamp-Lieb inequality and concentration property. In Section \ref{Section4}, we introduce twistings, associated semigroups and their generators. The goal of Section \ref{Section5} is to find conditions for these generators to be decomposable as a sum of a symmetric positive second order generator and a zero order potential. In Sections \ref{Section6} and \ref{Section7}, we obtain conditions to have intertwining relations for the $L^2$ semigroups on $1$-forms. Theorem \ref{P6.2} is a generalization of Theorem 2.2 in \cite{ABJ}, in a manifold setting, with the same kind of assumptions: conditions of symmetry and positiveness of the second order operator and bound on the potential. Theorem \ref{P7.1} extends this result when the second order operator is not symmetric non-negative. We achieve to release all assumptions over the second order operator by a stronger bound on the potential. These intertwinings are applied in Theorems~\ref{P6.3} and~\ref{P7.3} to obtain generalized Brascamp-Lieb and Poincaré inequalities. We finish, in Section \ref{Section8} with several illustrations of measures which fail Bakry-\'Emery criterion in different ways and for which our method brings bounds on the spectral gap. It improves the previously known lower bound for the classical Cauchy measure.  
%%%%%%%%%%%%%%%%%%%%%%%%%%%%%%%%%%%%%%%%%%%%%%%%%%%%%%%%%%%%%%%%%%%%%%%%%%%%
\section{Deformed parallel transport and Commutation}\label{Section2}
\setcounter{equation}0
%%%%%%%%%%%%%%%%%%%%%%%%%%%%%%%%%%%%%%%%%%%%%%%%%%%%%%%%%%%%%%%%%%%%%%%%%%%%
On a connected complete Riemannian manifold $(M,\langle\cdot,\cdot\rangle)$, endowed with its Levi-Civita connection $\nabla$, let $\CC^{\infty}(M)$ be the space of smooth real-valued functions and~$\CC^{\infty}_c(M)$ its subspace of compactly supported functions. In this paper, we consider the second order diffusion operator defined on $\CC^{\infty}(M)$ by
\begin{equation}\label  {E2.1}
Lf=\Delta f -\langle\nabla V, \nabla f\rangle,
\end{equation}
where $V$ is a smooth potential. We denote by $\mu$ the measure on $M $ with density~$e^{-V}$. On $\CC^\infty_c(M)$, the operator $L$ is symmetric with respect to~$\mu$, that is for all~$f,g \in \CC^\infty_c(M)$,
\begin{equation}
\int_M{Lf g\, d\mu} = -\int_M{\langle df, dg\rangle\, d\mu} = \int_M{f Lg\, d\mu}.
\label{E2.9}
\end{equation}

Let $X_t^x$ be a diffusion process with generator $L$, started at $x\in M$. Such a process exists and is unique in law, up to an explosion time $\tau_x$. We define a family~$(\BP_t)_{t\geq0}$ of operators on the space of bounded continuous functions by:   
\begin{equation}
\BP_tf(x) = \E\left[f\left(X_t^x\right)\ind_{t<\tau_x}\right].
\label{E2.8}
\end{equation}
The Markov property for diffusion processes implies that $(\BP_t)_{t\geq0}$ is a semigroup and for all $f\in\CC^{\infty}_c(M)$, we have 
\begin{equation}\label{E2.10}
\partial_t \BP_tf(x) = L\BP_tf(x) = \BP_tLf(x).
\end{equation}

Above the process $X_t^x$, one can construct the parallel translation $\sslash_t$. In a geometric point of view, it is an isometric isomorphism from $T_xM$ to $T_{X_t^x}M$, as the $\CC^1$ parallel translation, but in a stochastic point of view, it can be seen as a diffusion on the tangent bundle. Its generator on $1$-forms is given by :
\begin{equation}\label  {E2.6}
L^\sslash\alpha = \Delta^h \alpha-\nabla_{\nabla V}\alpha,\, \forall\alpha\in \Gamma(T^*M)
\end{equation}
where $\Delta^h$ is the horizontal Laplacian on $1$-forms. According to the Weitzenböck formula, for all $1$-form $\alpha$ and $w\in TM$, we have:
\begin{equation}\label  {E2.6.1}
\langle \Delta^h\alpha,w\rangle=\langle \square \alpha,w\rangle+\langle \alpha,\Ric(w)\rangle ,
\end{equation}
with $\square$ being the Hodge-de Rham Laplacian. For more details on the construction of this object, one can look at \cite{Hsu}.

The parallel translation is the first step to define a more relevant translation, in terms that will be explained bellow : the deformed parallel translation $W_t$. It is the linear map $T_xM\to T_{X_t^x}M$ determined by the differential equation:
\begin{equation}\label  {E2.2}
\left\{\begin{aligned} D_t W_tv &=  - \MM^*W_tv dt\\ W_0 &= \id_{T_xM}\\ \end{aligned}\right. ,
\end{equation}
where 
\begin{equation}
D_tW_tv=\sslash_t d\left(\sslash_t^{-1}W_tv\right)
\label{E2.2.2}
\end{equation}  
stands for the covariant derivative of~$W_tv$ and the operator $\MM^*$ is a section of ~$\End(TM)$ defined by \begin{equation}
\MM^* w=\nabla_w\nabla V+\Ric(w),\, \forall w\in TM .
\label{E2.4.2}
\end{equation}
Its adjoint operator, a section of $\End(T^*M)$ is denoted $\MM$. Remarks that both translations ~$\sslash_t$ and $W_t$ depend on the initial point $x$ but we avoid any reference to it when it is obvious. As an alternative definition or a major property,  Theorem $2.1$ in \cite{ACT} shows that for all $x\in M$ and $v\in T_x M$, $W_tv$ is the spatial derivative of a flow of the diffusion with generator $L$, obtained from $X_t^x$ by parallel coupling. In some way, the processes $X_t^x$ and $W_t$ are intertwined. As the parallel translation, the deformed parallel translation can be seen as a diffusion on the tangent bundle.
\begin{prop}
\label{P2.1}
The generator on $1$-forms of the deformed parallel translation is given by :
\begin{equation}
L^W\alpha =  L^\sslash\alpha  - \MM\alpha,\, \forall\alpha\in \Gamma(T^*M).
\end{equation}  
\end{prop}

\begin{proof}
This result is just an application of Ito formula to $1$-form (see \cite{LiXM1}). Let us detail a bit. Let $\alpha$ be a $1$-form and $v\in T_xM$, we have :
\begin{align*}
d\langle \alpha,W_tv\rangle 
&= d\langle \alpha \sslash_t, \sslash_t^{-1} W_tv\rangle\\
&= \left\langle d\left(\alpha\sslash_t\right),\sslash_t^{-1}W_tv\right\rangle + \left\langle\alpha\sslash_t, d\left(\sslash_t^{-1}W_tv\right)\right\rangle + \left\langle d(\alpha\sslash_t), d(\sslash_t^{-1}W_t)\right\rangle\\ 
&= \langle D_t\alpha,W_tv\rangle + \langle\alpha, D_tW_tv\rangle + \langle D_t\alpha, D_t W_t\rangle,\\
\end{align*}
where $D_t\alpha=d\left(\alpha\sslash_t\right)\sslash_t^{-1}$ stands for the covariant differential of~$\alpha$ along $X_t^x$. 
As parallel translation is a diffusion with generator $L^{\sslash}$, we have :
$$\langle D_t\alpha,W_tv\rangle \stackrel{(m)}{=} \langle L^{\sslash}\alpha, W_tv\rangle dt,$$ 
where $\stackrel{(m)}{=}$ means "up to a local martingale". As $W_t(x)$ satisfies equation \eqref{E2.2}, we obtain the second term and the quadratic term~$\langle D_t\alpha, D_t W_t\rangle$ vanishes as $D_tW_t$ has finite variation.
\end{proof}

For now, the operator $L^W$ has been defined only on smooth $1$-forms. We extend it in a~$L^2$-sense.  Let $L^2(\mu)$ be the space of measurable $1$-forms $\alpha$ such that 
\begin{equation*}
\int_M{|\alpha|^2\, d\mu} <+\infty .
\end{equation*}
Li proved in \cite{LiXM1} the following result.
\begin{theo}
The operator $L^W$ is essentially self-adjoint on $L^2(\mu)$.
\end{theo}

We give some ideas of the proof.
\begin{proof}
We denote by $\delta_V$ the adjoint of the exterior derivative on forms for the scalar product on $L^2(\mu)$. Some calculation shows that, for all smooth compactly supported $1$-forms $\alpha$, we have:
\begin{equation}
L^W\alpha=-(d\delta_V + \delta_Vd)\alpha
\label{E2.11}
\end{equation}
Then $L^W$ is unitary equivalent to a Witten Laplacian and so is essentially self-adjoint (see \cite{Hel} for more details).
\end{proof}

Then, without any assumptions, the deformed parallel translation defines a semigroup $(\BQ_t)_{t\geq0}$ on $L^2(\mu)$. We will see in Section \ref{Section3} that under suitable conditions on the potential $\MM$, it also generate a $\CC^0$ semigroup, on bounded continuous $1$-forms with a stochastic representation as \eqref{E2.8}. Remark that a continuous bounded $1$-form is not bounded as an element of $\CC^0(TM)$ and this is the major obstruction to the definition of a $\CC^0$ semigroup.

The generator of the deformed parallel translation satisfies a commutation formula. For all $f\in\CC^{\infty}(M)$, one has:
\begin{equation}\label  {E2.5}
dLf=L^Wdf.
\end{equation}
This is an intertwining relation at the level of generators. This commutation formula on generators and the intertwining relation at the level of stochastic processes suggest an intertwining relation between the semigroups $\BP$ and $\BQ$. The well-known following calculation is a motivation to obtain this kind of relation. We assume that $\mu$ is a probability measure. Then, it makes sense to look forward bounds on it variance. We also assume that the diffusion is ergodic i.e for all~$f\in\CC_c^\infty(M)$ :
\begin{equation}
\lim_{t\to+\infty}\BP_t f=\mu(f)\,\text{a.s} . 
\end{equation}
Remark that it is a weaker notion of ergodicity than the usual $L^2$ one.

\begin{prop}\label{P2.3}
For all $f,g\in \CC_c^\infty(M)$ we have the following covariance representation:
$$\Cov_\mu(f,g) =\int_0^{+\infty}\left(\int_M \langle df, d \BP_tg \rangle\, d\mu \right)\,dt.$$
\end{prop}

\begin{proof}
Using the ergodicity assumption and the relation \eqref{E2.10}, for all $f,g\in\CC^\infty_c(M)$ we have :
\begin{align*}
\Cov_\mu(f,g)
&=\int_M f(g-\mu(g))\, d\mu \\
&=\lim_{t\to+\infty}\int_M	f(g-\BP_t(g))\, d\mu\\
&=-\int_M \int_0^{+\infty}f L\BP_tg\, dt\, d\mu \\
&=\int_0^{+\infty}\left(\int_M \langle df, d \BP_tg \rangle\, d\mu \right)\,dt \\
\end{align*}
\end{proof}
This covariance representation enlightens the necessity of understanding the differential $d\BP_t$. 

%%%%%%%%%%%%%%%%%%%%%%%%%%%%%%%%%%%%%%%%%%%%%%%%%%%%%%%%%%%%%%%%%%%%%%%%%%%%
\section{A covariance inequality}\label{Section3}
\setcounter{equation}0
%%%%%%%%%%%%%%%%%%%%%%%%%%%%%%%%%%%%%%%%%%%%%%%%%%%%%%%%%%%%%%%%%%%%%%%%%%%%
The main goal of this section is to motivate the use of intertwining relation for functional inequalities and to enlighten the role of Bakry-\'Emery criterion.  As we are looking for a generalisation of this criterion, we want to see how it works. This section is also the opportunity to give a proof of the intertwining relation for the $\CC^0$ semigroups, using only the stochastic tools presented in Section \ref{Section2}. We obtain an asymmetric Brascamp-Lieb inequality in the spirit of Ledoux (see \cite{Led} or \cite{CC-EL}) and a gaussian concentration result. 

Firstly, we have to find a condition so as to properly define the semigroup. As an endomorphism of $T^*_xM$, the operator $\MM(x)$ defined in \eqref{E2.4.2}, is symmetric with respect to the metric. We denote by $\rho(x)$ the smallest eigenvalue of $\MM(x)$ and by $\rho$, its infimum over $M$:
\begin{equation}
\rho=\inf_{x\in M}\{\text{smallest eigenvalue of}\, \MM(x)\}
\end{equation}

The assumption of this section is the Bakry-Émery criterion (also known as the~${CD(\rho, \infty)}$ condition in \cite{BJL}).

\begin{hypothesis}[Bakry-Émery criterion]
The operator $\MM$ is uniformly bounded from below, i.e $\rho >-\infty$.
\end{hypothesis}

It is a sufficient condition for hypercontractivity of the diffusion and allows to prove Poincaré or Log-Sobolev inequalities (see \cite{BE}). Bakry proves in \cite{Bak} that, under this criterion, the diffusion $X$ does not explode (i.e for all $x\in M$, $\tau_x=+\infty$ almost surely). It is not a necessary condition. The following intertwining results are proved in \cite{LiXM1} with a finite moment criterion, weaker but less handy, and for other flow (see also the concept of $p$-completeness in \cite{LiXM2}), or in \cite{EL1} with finer bounds. The following result is well known.

\begin{prop}\label{P3.1}
Under the Bakry-Émery criterion, the semigroup $\BQ$ has the stochastic representation : for all bounded $1$-form $\alpha$, for all $x\in M$, for all $v\in T_v M$,
\begin{equation}\label{E3.1}
\langle\BQ_t\alpha,v\rangle = \E\left[\langle\alpha,W_tv\rangle\right],
\end{equation} 
and we have: for all $1$-form $\alpha$, for all $x\in M$,
\begin{equation*}
\|\BQ_t\alpha\|_\infty\leq e^{-\rho t}\|\alpha\|_\infty .
\end{equation*}
\end{prop}

\begin{proof}
The heart of the proof is to show that under this criterion, the deformed parallel translation is bounded. For all $x\in M$ and all $v\in T_xM$, one has 
\begin{align*}
d |W_tv|^2
&= 2\left\langle W_tv, D_tW_tv\right\rangle\\
&= -2\left\langle W_tv, \MM^* W_tv\right\rangle dt\\
&\leq -2\rho |W_tv|^2 dt .
\end{align*} 

By Grönwall lemma, this yields 
\begin{equation}\label{E3.2}
|W_tv|\leq e^{-\rho t}|v|,\, a.s.
\end{equation}
This shows that the stochastic representation \eqref{E3.1} is well-defined and concludes the proof. 
\end{proof}

\begin{prop}\label{P3.2}
Under the Bakry-Émery criterion, the semigroups $\BP$ and $\BQ$ are intertwined by the derivative of functions: for all $f\in\CC^{\infty}_c(M)$,
\begin{equation}
d \BP_tf = \BQ_t df.
\end{equation}
\end{prop}

\begin{proof}
Let $x\in M$, $v\in T_xM$ and $\gamma : I\to M$ a smooth curve such that $\gamma(0) =x$ and $\gamma'(0) =v$. According to \cite{ACT}, there exists a flow $X_t(a)$ of the $L$-diffusion such that $X_t(0)=X_t^x$, $X_0(a) =\gamma(a)$ and $\partial_aX_t(a) = W_t(a)\gamma'(a)$ where $W_t(a)$ is the deformed parallel translation above $X_t(a)$. For any $f\in \CC_c^\infty(M)$ and $t>0$, we have :
\begin{align*}
\langle d\BP_tf, v\rangle 
&=\left.\frac{d}{da}\right|_{a=0} \BP_t(\gamma(a))\\
&=\left.\frac{d}{da}\right|_{a=0} \E\left[f(X_t(\gamma(a)))\right]\\	
\end{align*}
	
The bound \ref{E3.2}, independent on the initial condition, and the regularity of $f$ guarantee the differentiation under the expectation. We have :
\begin{align*}
\langle d\BP_tf, v\rangle  
&= \E\left[\langle df, W_tv\rangle\right]\\
&= \langle \BQ_t df, v\rangle \\
\end{align*}
\end{proof}

This result has also been proved for $q$-form in \cite{ELL}. Armed with it, it is possible to obtain several result in analysis. Amongst them, there are  finiteness results of volume and homotopy group (see \cite{LiXM3}). Now, we are going to use this intertwining relation to obtain functional inequalities, in the spirit of Section \ref{Section6} and \ref{Section7}. We get back to the assumptions of ergodicity and finite measure from Proposition \ref{P2.3}. To begin with, we can rewrite the integral representation of the covariance using the intertwining. For all $f,g\in \CC_c^\infty(M)$, we have :
\begin{equation}\label{E3.3}
\Cov_\mu(f,g) =\int_0^{+\infty}\int_M\langle df, \BQ_t(dg)\rangle \, d\mu\, dt.
\end{equation}
It is the key argument to prove the following asymmetric Brascamp-Lieb inequality. This inequality is called asymmetric because it gives an $L^1$-$L^\infty$ bound of the covariance. 
 
\begin{theo}[Asymmetric Brascamp-Lieb inequality]\label{P3.3}
Assume that $\rho>0$, then for all functions $f$, $g\in\CC_c^\infty(M)$, one has
\begin{equation*}
|\Cov_\mu(f,g)|\leq\frac{1}{\rho}\|dg\|_{\infty}\int_{M}|df|\, d\mu.
\end{equation*}
\end{theo}

\begin{proof}
From Hölder inequality, for all $f,g\in \CC_c^\infty(M)$ and $t\geq0$, we have :
$$\left|\int_M\langle df, \BQ_t(dg)\rangle\, d\mu\right|\leq \|\BQ_t(dg)\|_\infty\int_{M}|df|\, d\mu.$$
Using the bound from Proposition \ref{P3.1}, we have :
$$|\BQ_t(dg)\|_\infty \leq e^{-\rho t}\|dg\|_{\infty}.$$
With the representation \eqref{E3.3} of the covariance, it ends the proof.
\end{proof}

Remark that in \cite{Bak}, the assumption of finiteness of $\mu$ is proven to be implied by the positivity of $\rho$. 

A consequence of Theorem \ref{P3.3}, is the Gaussian concentration of the probability~$\mu$. This concentration result has been shown by Ledoux in \cite{Led2} for the volume measure of a compact Riemannian manifold under the condition of positive Ricci curvature and in \cite{Led3} in the Euclidean space under the condition of strictly convex potential. This inequality is deeply exposed in \cite{Led1}. Our proof gives a new outlook of the result, with only stochastic tools.   

\begin{prop}\label{P3.4}
If $\rho>0$, then for all $1$-Lipschitz $f\in\CC_c^\infty(M)$ and for all $r>0$,
\begin{equation}
\mu\left(|f-\mu(f)|>r\right)\leq 2e^{-\rho\frac{r^2}{2}}.
\end{equation}
\end{prop}

\begin{proof}
The idea of the proof is to bound the Laplace transform. Let $f$ be a smooth compactly supported $1$-Lipschitz function. Without any loss of generality, we can assume that $f$ is centred. For any $\lambda >0$, we have:
\begin{align*}
\frac{d}{d\lambda}\E_\mu[e^{\lambda f}] 
&= \Cov_\mu(f,e^{\lambda f})\\
&\leq \frac{1}{\rho}\|df\|_\infty \int_M\lambda|df|e^{\lambda f}\, d\mu \\
&\leq \frac{\lambda}{\rho}\E_\mu\left[e^{\lambda f}\right] .\\
\end{align*}
By Grönwall lemma, it yields :
\begin{equation}
\E_\mu[e^{\lambda f}] \leq e^{\lambda^2/2\rho}.
\end{equation}
The proof ends by using Markov's inequality and optimizing in $\lambda$.
\end{proof}  
%%%%%%%%%%%%%%%%%%%%%%%%%%%%%%%%%%%%%%%%%%%%%%%%%%%%%%%%%%%%%%%%%%%%%%%%%%%%
\section{Twisted processes and semigroups}\label{Section4}
\setcounter{equation}0
%%%%%%%%%%%%%%%%%%%%%%%%%%%%%%%%%%%%%%%%%%%%%%%%%%%%%%%%%%%%%%%%%%%%%%%%%%%%
Let $B$ be a smooth section of $\GL(TM)$, i.e for all $x\in M$, $B(x)$ is an isomorphism of $T_xM$. The section~$B$ is used to twist the semigroup so as to obtain an intertwining relation even when the Bakry-Émery criterion is not satisfied. In this section, we are going to construct the tree levels, process, generator and semigroup, and prove a commutation at the level of generators. The intertwining relation at the level of semigroups will be treated in Section \ref{Section6} and \ref{Section7}. Firstly, we define the $B$-parallel translation above $X^x$ by conjugation as: 
\begin{equation}\label{PB}
\sslash^B_t=B(X_t^x)\sslash_tB(x)^{-1}: T_xM\to T_{X_t^x}M .
\end{equation}
This new translation is also a diffusion on $TM$ and we can calculate its generator on $1$-forms, denoted by $L^{\sslash,B}$. 
\begin{prop}\label{P4.1}
The generator on $1$-forms of the $B$-parallel translation is given by 
\begin{equation*}
L^{\sslash,B}\alpha=L^\sslash\alpha+ 2(B^{-1})^*\nabla B^* \cdot\nabla \alpha+(B^{-1})^*(L^\sslash B^*)\alpha,\, \forall\alpha\in \Gamma(T^*M).
\end{equation*}
with the contraction $\nabla B^* \cdot\nabla \alpha = \sum_i \nabla_{e_i} B^* \cdot\nabla_{e_i} \alpha$ for any orthonormal basis $(e_i)_i$.
\end{prop}

\begin{proof}
For all $1$-form $\alpha$, $w\in T_xM$ and $t\geq0$, we have :
\begin{align*}
\langle \alpha, \sslash_t^Bw\rangle 
&= \langle\alpha, B(X_t^x)\sslash_t B^{-1}(x) w\rangle\\ 
&= \langle B^*\alpha, B(x)^{-1}w\rangle\\
\end{align*}
Using the diffusion property of $\sslash_t$, we have :
$$d\langle \alpha, \sslash_t^Bw\rangle \stackrel{(m)}{=}\langle L^{\sslash}(B^*\alpha), \sslash_t B^{-1}(x)w\rangle dt.$$
By definition of $L^{\sslash,B}$, we have :
$$d\langle \alpha, \sslash_t^Bw\rangle \stackrel{(m)}{=}\langle L^{\sslash,B}\alpha, \sslash_t^B w\rangle dt.$$
Together, it yields :
$$L^{\sslash, B}\alpha = (B^*)^{-1}L^\sslash \left(B^*\alpha\right).$$
This ends the proof.
\end{proof}

Unlike the parallel translation, the $B$-parallel translation is not an isometry for the Riemmanian metric. Actually, it is not adapted to the Riemannian metric. To get back to a notion of isometric translation along curves, we need to twist the metric too and use the $B$-twisted metric: for all~$v,w\in T_xM$
\begin{equation}
\langle v, w\rangle_B = \langle B^{-1}(x)v, B^{-1}(x)w\rangle .
\end{equation}

However, the twisted-parallel translation $\sslash^B$ is still not the Levi-Civita parallel translation associated to the $B$-metric but we have a simple relation between them. Let us denote $\nabla^B$ the connexion associated to $\sslash^B$. It satisfies :
\begin{equation}
\nabla^B = \nabla -(\nabla B)B^{-1}.
\end{equation}
Its torsion $T^B$ is generically non-vanishing and satisfies : for all $X,Y\in\Gamma(TM)$,
\begin{equation}
T^B(X,Y) = (\nabla_YB)B^{-1}X-(\nabla_XB)B^{-1}Y.
\end{equation}
On the other hand, the Levi-Civita connexion $\nabla^{B\LL\CC}$, satisfies : for all $X,Y,Z\in\Gamma(TM)$,
\begin{equation}
\left\langle\nabla^{B\LL\CC}_XY, Z\right\rangle_B = \left\langle \nabla^B_XY -\frac{1}{2} T^B(X,Y), Z\right\rangle + \frac{1}{2}\left\langle T^B(X,Z), Y\right\rangle_B + \frac{1}{2}\left\langle T^B(y,Z), x\right\rangle_B.
\end{equation} 
Hence both connexions differ not only from a torsion term but also from a symmetric part. It is not a torsion skew symmetric case (see \cite{ELL}). This also means that a twist is not only a change of metric. 

Now, as in the non-twisted case, the next step is to define the $B$-deformed parallel translation as:
\begin{equation}\label{WB}
W^B_t=B(X_t^x)W_t(x)B^{-1}: T_xM\to T_{X_t^x}M.
\end{equation}
This definition uses the previous definition of the deformed parallel translation, but, as a straightforward calculation shows, $W_t^B$ could have been defined as in~\eqref{E2.2} by a stochastic covariant equation :
\begin{equation}\label{E4.9}
D_t^BW_t^B = -\MM_B^*W_t^B
\end{equation}
where $D_t^B$ stands for the $B$-covariant derivative $\sslash_t^B d\left({\sslash_t^B}^{-1}\right)$ and $\MM_B^*$ is the adjoint of 
\begin{equation}\label  {E4.6}
\MM_B=(B^*)^{-1} \MM B^*.
\end{equation}

Again, this translation is a diffusion in $TM$. As for $\sslash_t^B$, the same calculation shows that its generator on $1$-forms, $L^{W,B}$, is conjugated to the generator $L^W$ : $$L^{W,B}= (B^*)^{-1}L^W B^*.$$ This gives a first decomposition of $L^{W,B}$.
\begin{prop} \label{P4.2}
The $B$-deformed parallel translation is a diffusion with generator on $1$-forms  
\begin{equation}
L^{W,B}=L^{\sslash,B}-\MM_B
\end{equation}
\end{prop}
  
The argument of conjugacy shows that $L^{W,B}$ and~$L$ are intertwined by $(B^*)^{-1}d$:
\begin{equation}\label{E4.12}
(B^*)^{-1}d L = L^{W,B}(B^*)^{-1}d.
\end{equation}

The generator $L^{W,B}$ can be extended in a $L^2$-sense. We denote by $\langle\cdot,\cdot\rangle_B$ the intertwined-metric on $1$-forms:  for two $1$-forms $\alpha,\beta$, 
\begin{equation}\label{E4.20}
\langle \alpha,\beta\rangle_B= \langle B^* \alpha, B^* \beta\rangle ,
\end{equation}
and by $L^2(B,\mu)$ the space of measurable $1$-forms $\alpha$ such that
\begin{equation}
\int_M{|\alpha|^2_B\, d\mu} <+\infty.
\end{equation}
As $L^W$, $L^{W,B}$ is also essentially self-adjoint, on $L^2(B,\mu)$ and is associated to a~$L^2$ semigroup of diffusion on $1$-forms, $\BQ_t^B$. Under suitable conditions, it would generates a semigroup on smooth compactly supported $1$-forms, also denoted by~$\BQ_t^B$, with the stochastic representation
\begin{equation}
\langle \BQ_t^B\alpha,v\rangle =\E\left[\langle \alpha, W_t^Bv\rangle\ind_{t<\tau_x}\right].
\label{E4.7}
\end{equation}
 
Next proposition present a condition for the existence of the $\CC^0$ semigroup and intertwining.
\begin{prop}\label{P4.3}
Under the Bakry-Émery criterion, $\BQ^B$ is well-defined by the formula \eqref{E4.7} and is intertwined to $\BP$ by $(B^*)^{-1} d$, i.e. for all $f\in\CC^{\infty}_c(M)$, for all~$t\geq0$,
\begin{equation*}
(B^*)^{-1} d \BP_t f = \BQ_t^B\left( (B^*)^{-1} df \right).
\end{equation*}
\end{prop}

\begin{proof}
As in Proposition \ref{P3.1}, the Bakry-Émery criterion prove the existence of the stochastic representation \eqref{E4.7}. For all $f\in\CC^{\infty}_c(M)$, we have:
\begin{align*}
(B^*)^{-1}) d \BP_t f 
&= (B^*)^{-1} \BQ_t df\\
&= (B^*)^{-1} \E\left[\langle df,W_t\cdot\rangle\right]\\
&= \E\left[\langle df,W_tB^{-1}(x) \cdot\rangle\right]\\
&= \E\left[\langle (B^*)^{-1} df, B(X_t^x)W_tB^{-1}(x) \cdot\rangle\right]\\
&= \E\left[\langle (B^*)^{-1} df, W_t^B \cdot\rangle\right] .\\
\end{align*}
\end{proof}

This results is not very satisfying because the goal of twisting is to obtain results when the Bakry-\'Emery criterion is not proven. Furthermore, the potential $\MM^B$ appearing in \eqref{E4.9} is conjugated to $\MM$ so they have the same eigenvalues. Then $\MM^B$ does not seem useful to improve inequalities such as in Section \ref{Section3}, even if we could obtained the intertwining relation without using the Bakry-Émery criterion. In order to find a more relevant potential, we get back to the definition of $W_t^B$. 

\begin{prop}\label{P4.4}
The $B$-deformed parallel translation satisfies the stochastic covariant equation :
\begin{equation}
D_tW_t^B = - \left(\MM_B -L^\sslash(B)B^{-1}\right)W_t^B v\, dt + \left(\nabla_{d_mX_t^x}B\right)(B^{-1})^*W_t^B v,
\end{equation}
where $d_mX_t^x$ is the martingale part of the Ito derivative of $X_t^x$.
\end{prop}
\begin{proof}
For all $x\in M$; $v\in T_xM$ and $t\geq0$, we have :
\begin{align*}
D_t W_t^Bv
&=\sslash_t d\left(\sslash^{-1}_t W_t^Bv\right)\\
&=\sslash_t d\left(\sslash^{-1}_tB(X_t^x)\sslash_t \sslash_t^{-1} W_tB(x)^{-1}v\right)\\
&=\sslash_t d\left(\sslash^{-1}_tB(X_t^x)\sslash_t\right)\sslash_t^{-1}W_tB(x)^{-1}v + B(X_t^x)D_t W_tB(x)^{-1}v +0\\
\end{align*}
In the first term, we recognize the stochastic covariant derivative in $\End(TM)$ : $$D_tB(X_t^x)= \sslash_t d\left(\sslash^{-1}_tB(X_t^x)\sslash_t\right)\sslash_t^{-1}.$$ The second term is given by \eqref{E2.2}. There is no quadratic term as $W_t$ has finite variations. Then we have :
\begin{equation*}
D_t W_t^Bv = \left(L^\sslash B(X_t^x)dt+\nabla_{d_mX_t^x}B(X_t^x)\right)W_tB(x)^{-1}v -B(X_t^x)\MM W_t B(x)^{-1}v dt.
\end{equation*}	
\end{proof}

The potential involved in the finite variations part of $D_t  W_t^B$ is a new one. We denote it by
\begin{equation}
M_B=\MM_B -(B^*)^{-1}L^\sslash(B^*).
\label{E4.10.1}
\end{equation}
Thinking to the calculation of Section \ref{Section3}, this potential seems more relevant in a stochastic point of view : the growth of $|W_t^B|^2$ will be controlled by its first eigenvalue (but not only as there is a martingale part). This new potential leads to a second decomposition of the generator $L^{W,B}$ :
\begin{equation}
L^{W,B}= L^\sslash_B - M_B.
\label{E4.11}
\end{equation}
 where $L^\sslash_B$ is defined by :
\begin{equation}
L^\sslash_B = L^\sslash +2(B^*)^{-1}\nabla B^*\cdot\nabla\cdot 
\label{E4.10}
\end{equation}

From the operators point of view, this new split seems more satisfying too. Indeed, the potential $M_B$ contains all the zero-order terms of $L^{W,B}$ and only them. Of course, this potential is also more relevant in an heuristic way, if we think about the Euclidean study \cite{ABJ}. It is also linked to the work on Lyapunov functions in \cite{BCG} and \cite{CGWW} as we will see in Section \ref{Section5}. So, this potential is the natural candidate for a generalization of Bakry-Émery criterion.
%%%%%%%%%%%%%%%%%%%%%%%%%%%%%%%%%%%%%%%%%%%%%%%%%%%%%%%%%%%%%%%%%%%%%%%%%%%%
\section{Symmetry and positiveness of $-L^\sslash_B$}\label{Section5}
\setcounter{equation}0
%%%%%%%%%%%%%%%%%%%%%%%%%%%%%%%%%%%%%%%%%%%%%%%%%%%%%%%%%%%%%%%%%%%%%%%%%%%%
First, as we noticed, $L^{W,B}$ is conjugated to $L^W$, and so, is self-adjoint in $L^2(B,\mu)$. For the same reason, in the subspace of  twisted gradients ${\{(B^*)^{-1}df : f\in\CC_c^\infty(M)\}}$, we additionally have the non-positiveness of $L^{W,B}$:
\begin{align*}
\int_M{\langle (B^*)^{-1}df, L^{W,B}(B^*)^{-1}df\rangle_B\, d\mu} 
&= \int_M{\langle df, L^{W}df\rangle\, d\mu}\\
&=\int_M{\langle df, d (Lf)\rangle\, d\mu}\\
&=-\int_M{(Lf)^2\, d\mu}.\\
\end{align*}

The classical result, in non-twisted cases, use the decomposition of $L^W$ as the sum of a symmetric non-positive operator and a potential bounded from below. So we are looking for conditions such that $L^\sslash_B$ is symmetric with respect to the $B$-twisted metric. This is not trivial, even in the subspace of twisted gradients. First, from integration by parts for the horizontal Laplacian, we have 
\begin{equation}\label{E5.3}
\int_M\langle (-L^\sslash)\alpha,\beta\rangle \,d\mu=\int_M \langle \nabla\alpha,\nabla\beta\rangle\,d\mu
\end{equation}
with $\langle \nabla\alpha,\nabla\beta\rangle =\sum_i\left\langle \nabla_{e_i}\alpha,\nabla_{e_i}\beta\right\rangle$, with $(e_i)_i$ any orthonormal basis. Then, on one hand, we have :
\begin{align*}
\int_M{\langle (-L^\sslash) \alpha, \beta \rangle_B\, d\mu} 
&= \int_M{\langle (-L^\sslash) \alpha, (B^*)^tB^*\beta \rangle\, d\mu} \\
&= \int_M{\langle \nabla \alpha, \nabla((B^*)^tB^*\beta) \rangle\, d\mu} \\ 
&= \int_M{\langle \nabla \alpha, (B^*)^tB^*\nabla\beta \rangle\, d\mu} + \int_M{\langle \nabla \alpha, \nabla((B^*)^tB^*)\beta \rangle\, d\mu}\\
&= \int_M{\langle \nabla \alpha, \nabla\beta \rangle_B\,d\mu} +\int_M{\langle \nabla \alpha, \nabla(B^*)^tB^*\beta \rangle\, d\mu}\\
&\quad +\int_M{\langle \nabla \alpha, (B^*)^t \nabla(B^*)\beta \rangle\, d\mu}\\
\end{align*}
where $(B^*)^t$ denotes the dual map of $B^*$ with respect to scalar products on $T^* M$. On the other hand, we have : 
\begin{align*}
-\int_M\left\langle 2(B^{-1})^*\nabla B^* \cdot\nabla\alpha, \beta\right\rangle_B\,d\mu
&=-2\int_M\left\langle \nabla B^* \cdot\nabla\alpha, B^*\beta\right\rangle\,d\mu\\
&=-2\int_M\left\langle \nabla\alpha,\left(\nabla B^*\right)^tB^* \beta\right\rangle d\mu.\\
\end{align*}

This yields
\begin{equation}
\int_M{\left\langle \left(-L^\sslash_B\right) \alpha, \beta \right\rangle_B\, d\mu} = \int_M{\left\langle \nabla \alpha, \nabla\beta \right\rangle_B\,d\mu} -\int_M{\left\langle B^*\nabla\alpha,\BB (B^*\beta)\right\rangle\, d\mu}.
\label{E5.5}
\end{equation}

where
\begin{equation}
\BB=\left((\nabla B^*)(B^*)^{-1}\right)^t-(\nabla B^*)(B^*)^{-1}
\label{E5.6}
\end{equation}

We immediately get this first criterion of symmetry and non-negativeness.
\begin{prop}\label{P5.1}
If $\BB=0$, then the generator $-L_B^\sslash$ is symmetric with respect to the $B$-metric, non-negative and we have:
\begin{equation}
\label{E5.7}
-\int_M\langle L_B^\sslash\alpha,\beta\rangle_B \,d\mu=\int_M \langle \nabla\alpha,\nabla\beta\rangle_B\,d\mu
\end{equation}
\end{prop}

The criterion of Proposition \ref{P5.1} is obviously not necessary but it gives a condition easy to check and not to constraining as we will see. 

On other way to find a condition of symmetry is to look at the potential rather than the operator. The operator $L^{W,B}$ and the potential $\MM_B$ are symmetric with respect to the $B$-metric and we have :
\begin{equation}
\label{E5.8}
L_B^\sslash=L^{W,B}+\MM_B - (B^*)^{-1} L^\sslash (B^*) .
\end{equation}
So a necessary and sufficient condition for the $B$-symmetry of~$L^\sslash_B$ is the $B$-symmetry of the potential $(B^*)^{-1} L^\sslash (B^*)$. But unlike the condition of Proposition~\ref{P5.1}, this is not a sufficient condition for positiveness. For example, one can look at $(\R^*_+)^2$ with the potential $V(x,y)=x+y$ and the twist
$$B^*=\left(\begin{array}{cc}
	\varphi & \varphi\\ 
	1 & e^V\\
\end{array}\right),$$
where $\varphi$ is positive such that $L\varphi\neq0$. The associated $L^\sslash_B$ is symmetric but is not non-negative. 

The following result is immediate and gives examples satisfying the condition of Proposition \ref{P5.1}.
\begin{prop}\label{P5.2}
If $B(x)=b(x)\id_{T_xM}$ for some $b\in\CC^{\infty}(M)$, then ~$\BB=0$.
\end{prop}

\begin{proof}
If $B(x)=b(x)\id_{T_xM}$ then we have $B^* (x)=b(x)\id_{T_x^* M}$, $\nabla B^* =\nabla b\otimes \id_{T^* M}$ and~${(\nabla B^*)(B^*)^{-1}=b^{-1} \nabla b\otimes \id_{T_x^* M}}$. It is clearly symmetric.
\end{proof}

In the same spirit, we have a result for product manifolds endowed with a product-compatible metric.
\begin{prop}
Let $(M_i,g_i)_{1\leq i\leq n}$ be Riemmanian manifolds and $(M, g)$ the product manifold endowed with the product metric. If $B(x) = \sum_{i=1}^n b_i(x) \id_{T_{x_i}M_i}$ for some functions $b_i\in\CC^{\infty}(M)$, then $\BB=0$. 
\end{prop}

This results allows us to consider non-homothetic diagonal twists in $\R^n$. Keep in mind that for non-product manifolds, this result may be invalid. For example, consider the Heisenberg group $\HH =(\R^3, *)$, where 
\[(x,y,z)*(\hat{x},\hat{y},\hat{z})= (x+\hat{x}, y+\hat{y}, z+\hat{z}+\frac{1}{2}(x\hat{y}-\hat{x}y)),\]
endowed with the left-invariant metric, such that the following vector fields form an orthonormal frame at each point :
\[e_1= \partial_x-\frac{y}{2}\partial_z,\quad e_2 =\partial_y +\frac{x}{2}\partial_z,\quad e_3=\partial_z. \] 
In this space, a straightforward calculation proves that the only diagonal twists satisfying $\BB=0$ are homothetic.

It seems difficult to find other kinds of examples of twist satisfying the criterion~$\BB=0$. Nevertheless, this class of twists is directly linked to the study of Lyapunov functions in \cite{CGWW}. Actually, for a twist $B(x)=b(x)\id_{T_xM}$, the eigenvalues of $M_B$ are the eigenvalue of $\MM$ shifted by $-b^{-1}L(b)$, which appears in their calculation. Finally, we will see in the last section that this class of twists allows us to treat various types of examples.  
%%%%%%%%%%%%%%%%%%%%%%%%%%%%%%%%%%%%%%%%%%%%%%%%%%%%%%%%%%%%%%%%%%%%%%%%%%%%
\section{Intertwining: a symmetric positive case}\label{Section6}
\setcounter{equation}0
%%%%%%%%%%%%%%%%%%%%%%%%%%%%%%%%%%%%%%%%%%%%%%%%%%%%%%%%%%%%%%%%%%%%%%%%%%%%
The goal of this section, is to prove the intertwining relation and Poincaré inequality under the assumption $\BB = 0$.  According to Proposition \ref{P5.1},  $-L^\sslash_B$ is symmetric, non-negative, with respect to $\langle\cdot,\cdot\rangle_B$. As $L^{W,B}$ is symmetric with respect to this metric, then $M_B$ is symmetric too. We denote by $\rho_B$ the infimum over $M$ of the smallest eigenvalue of $B^* M_B (B^*)^{-1}$:
\begin{equation}
\label{E6.1}
\rho_B=\inf_{x\in M}\left\{\hbox{smallest eigenvalue of }B^* M_B (B^*)^{-1} \right\}.
\end{equation}

We also assume that $\rho_B$ is bounded from below. As we already said, the generator $L^{W,B}$ is essentially self-adjoint. With this new assumption, $L^{W,B}$ is the sum of a symmetric non-negative operator $L^\sslash_B$ and a bounded from below potential $M_B$. So we could obtain a new proof of the the essential self-adjointness as a generalization of proof of Strichartz in \cite{Stri}. In order to obtain the intertwining relation, we need to show that~$(B^*)^{-1} d \BP_t f$ is the unique $L^2$ strong solution to the Cauchy problem 
$$\left\{\begin{aligned}
&\partial_tF = L^{W,B} F\\
&F(\cdot,0)=G\in L^2(B,\mu)\\
\end{aligned}\right.$$
where the mapping $t\mapsto F(\cdot,t)$ is continuous from $\R_+$ to $L^2(B,\mu)$. Remark that we are looking for a strong solution : in this Cauchy problem, $L^{W,B}$ has to be understood as a differential operator and not an $L^2$ operator. Actually, as we do not know the domain of $L^{W,B}$, we cannot use the uniqueness in the sense of self-adjoint operator.

\begin{prop}\label{P6.1}
Assume that $\BB=0$ and that $M_B$ is uniformly bounded from below. Let $F$ be a solution of the $L^2$ Cauchy problem above. Then, we have \begin{equation*}
F(\cdot,t)=\BQ_t^B(G), t\geq0.
\end{equation*}  
\end{prop}

\begin{proof}
We generalize the argument of \cite{Li} and \cite{ABJ} which deal respectively with the case of a Laplacian in a Riemannian manifold and the case of our operator $L^{W,B}$ in $\R^n$. By linearity, it is sufficient to show the uniqueness of the solution for the zero initial condition. Replacing the solution $F$ by $e^{-\rho_Bt}F$, let us assume that $M_B$ is positive semi-definite. For every $\phi\in\CC^{\infty}_c(M)$ and~$\tau>0$, we have:
\begin{align*}
\int_0^\tau\int_M{\phi^2\langle F, L^\sslash_B F\rangle_B\, d\mu\, dt}
&=\int_0^\tau\int_M{\phi^2\langle F, (L^{W,B}+M_B) F\rangle_B\, d\mu\, dt}\\
&=\int_0^\tau\int_M{\phi^2\frac{1}{2}\partial_t|F|^2_B\, d\mu\, dt} + \int_0^\tau\int_M{\phi^2\langle F,M_B F\rangle_B\, d\mu\, dt}\\
&\geq\int_M{\phi^2\frac{1}{2}|F(\cdot,\tau)|^2_B\, d\mu} .
\end{align*}

On the other hand, by the integration by parts formula of Proposition \ref{P5.1}, we have 
\begin{align*}
\int_0^\tau\int_M{\phi^2\langle F, L^\sslash_B F\rangle_B\, d\mu\, dt}
&=-\int_0^\tau\int_M{\langle\nabla(\phi^2 F), \nabla F\rangle_B\, d\mu\, dt}\\
&=-\int_0^\tau\int_M{\phi^2|\nabla F|^2_B\, d\mu\, dt}\\
&\qquad-2\int_0^\tau\int_M{\langle \nabla\phi\otimes F, \phi\nabla F\rangle_B\, d\mu\, dt} .\\
\end{align*}

By Cauchy-Schwarz inequality, we have for every $\lambda >0$, 
\begin{equation}
2\left|\langle \nabla\phi\otimes F, \phi\nabla F\rangle_B\right|\leq \lambda|\nabla \phi|^2_2|F|^2_B +\frac{1}{\lambda}\phi^2|\nabla F|^2_B .
\label{E6.2}
\end{equation}

Combining the above inequalities, in the particular case of $\lambda =2$, we obtain 
\begin{align*}
\frac{1}{2}\int_M{\phi^2|F(\cdot,\tau)|^2_B\, d\mu}&\leq -\frac{1}{2}\int_0^\tau\int_M{\phi^2|\nabla F|^2_B\, d\mu\, dt}\\ 
&\qquad + 2\int_0^\tau\int_M{|\nabla \phi|^2_2|F|^2_B\, d\mu\, dt} .
\end{align*}

By completeness of $M$, there exists a sequence of cut-off functions $(\phi_n)_{n\in\N}$ in~$\CC_c^\infty(M)$ such that $(\phi_n)_n$ converge to $1$ pointwise and $\|\nabla\phi_n\|_\infty\to 0$ as $n\to \infty$. Plugging this sequence in the previous inequality, gives 
\begin{equation}
\int_M{|F(\cdot,\tau)|^2_B\, d\mu}=0, \tau>0.
\label{E6.4}
\end{equation}
Hence $F=0$ in $\CC^0\left(\R_+, L^2(B,\mu)\right)$.
\end{proof}

\begin{theo}\label{P6.2}
Assume that $\BB=0$ and that $M_B$ is uniformly bounded from below. Then the semigroups $\BP$ and $\BQ^B$ are intertwined by $(B^*)^{-1}d$, i.e for all ${f\in\CC^\infty_c(M)}$ and $t\geq0$, we have :
\begin{equation*}
(B^*)^{-1} d\BP_tf = \BQ_t^B\left((B^*)^{-1} df\right).
\end{equation*}
\end{theo}

\begin{proof}
The main argument is to prove that $F(\cdot,t) = (B^*)^{-1} d\BP_tf$ is a solution of the previous $L^2$ Cauchy problem with initial condition $G=(B^*)^{-1} df$. First, $G$ is in $L^2(B,\mu)$ since $f$ is compactly supported. For all $t>0$, we have:
\begin{align*}
\int_M{|F(\cdot,t)|^2_B\, d\mu}
&=\int_M{|(B^*)^{-1} d\BP_tf|^2_B\, d\mu}\\
&=\int_M{|d\BP_tf|^2\, d\mu}\\
&=-\int_M{\BP_tf L\BP_tf\, d\mu},\\
\end{align*}
which is finite since $\BP_tf\in\DD(L)\subset L^2(\mu)$. So $F(\cdot,t)$ is in $L^2(B,\mu)$ for every $t>0$. Besides, the $L^2$ continuity is proven by the same calculation, since for every $t,s\geq0$,
\begin{equation}
\int_M{|F(\cdot,t)-F(\cdot,s)|^2_B\, d\mu}= -\int_M{(\BP_tf-\BP_sf) L(\BP_tf-\BP_sf)\, d\mu}.
\label{E6.5}
\end{equation}
By spectral theorem, this is upper bounded by $(\sup_{x\in\R_+}|\sqrt{x}(e^{-tx}-e^{-sx})|)^2\|f\|^2_2$ which tends to zero as $s$ tends to $t>0$ (see \cite{RS} for more details on spectral theorem). For the right-continuity in $t=0$, we use that 
\begin{equation}
\int_M{(\BP_sf-f) L(\BP_sf-f)\, d\mu} = \int_0^s\int_0^s\int_M{\BP_tLf\BP_uL^2f\, d\mu\, du\, dt}.
\end{equation}
Finally, the commutation property \eqref{E4.12}, yields
$$\partial_tF = (B^*)^{-1} dL\BP_tf = L^{W,B}\left((B^*)^{-1} d \BP_tf\right) = L^{W,B}F.$$
The result follow by the uniqueness of the solution of the Cauchy problem.
\end{proof}

With this intertwining relation, we are now able to prove some functional inequalities. We get back to the assumptions of ergodicity and finite measure from Proposition \ref{P2.3}. We have an integral representation of the covariance, using the semigroup $\BQ^B$ instead of $\BQ$: for all $f,g\in\CC_c^{\infty}(M)$, 
\begin{equation}
\Cov_\mu(f,g) = \int_0^{+\infty}\left(\int_M\left\langle (B^*)^{-1}df, \BQ_t^B((B^*)^{-1} dg)\right\rangle_B\, d\mu\right)\,dt .
\label{E6.6}
\end{equation}

The main application of this covariance's representation is a generalization of an inequality due to Brascamp and Lieb, in \cite{BL}, known as Brascamp-Lieb inequality. 

\begin{theo}[Generalized Brascamp-Lieb inequality - symmetric case]\label{P6.3}
Assume that $\BB=0$ and that $M_B$ is positive definite, then for every $f\in\CC^\infty_c(M)$, we have :
\begin{equation*}
\Var_\mu(f)\leq \int_M{\langle df,\left((B^*M_B(B^*)^{-1}\right)^{-1} df\rangle\, d\mu}.
\end{equation*}
\end{theo}

Firstly, we need a little lemma.
\begin{lemme}\label{P6.4}
Let $C$ and $D$ be symmetric non-negative operators such that~$D$ and~$C+D$ are invertible. Then we have
$$0\leq D^{-1} -(C+D)^{-1}.$$
\end{lemme}

\begin{proof}
We have:
\begin{align*}
D^{-1}-(C+D)^{-1} = (C+D)^{-1}CD^{-1},
\end{align*}
and we have 
\begin{align*}
\langle (C+D)^{-1}CD^{-1}\alpha,\alpha\rangle = \langle CD^{-1}\alpha,(C+D)^{-1}\alpha\rangle .
\end{align*}
Letting $(C+D)^{-1}\alpha=\beta$ this rewrites as 
\begin{align*}
\langle CD^{-1}(C+D)\beta, \beta\rangle&=\langle CD^{-1} C\beta,\beta\rangle +\langle C\beta,\beta\rangle\\
&=\langle D^{-1} C\beta,C\beta\rangle +\langle C\beta,\beta\rangle\geq 0,
\end{align*}
since $D^{-1}$ and $C$ are non-negative.
\end{proof}

\begin{proof}[Proof of Theorem \ref{P6.3}]
First, let assume that $\rho_B$ is positive. This implies that for all $1$-form $\alpha$, we have
\begin{equation}
\int_M{\langle (-L^{W,B})\alpha,\alpha\rangle_B\, d\mu}\geq \rho_B\int_M|\alpha|^2_B\, d\mu.
\label{E6.7}
\end{equation}
So $-L^{W,B}$ is essentially self-adjoint and bounded from below by $\rho_B \id$. Then it is invertible in $L^2(B,\mu)$ i.e given any smooth compactly supported $1$-form $\alpha$, the Poisson equation $-L^{W,B} \beta=\alpha$ admits a unique solution $\beta$ in the domain of $L^{W,B}$ which has the following integral representation:
\begin{equation}
\beta = \int_0^{+\infty}{\BQ_t^B(\alpha)\, dt} = (-L^{W,B})^{-1}\alpha .
\label{E6.8}
\end{equation}
Using the variance representation formula \eqref{E6.6}, we have 
\begin{align*}
\Var_\mu(f)
&=\int_0^\infty \left(\int_M \left\langle (B^*)^{-1}df, \BQ_t^{B}\left((B^*)^{-1}df\right)\right\rangle_B \,d\mu\right)\,dt\\
&=\int_M \left\langle (B^*)^{-1}df, \int_0^\infty \BQ_t^{B}\left((B^*)^{-1}df\right)\,dt\right\rangle_B \,d\mu\\
&=\int_M \left\langle (B^*)^{-1}df, (-L^{W,B})^{-1}\left((B^*)^{-1}df\right)\right\rangle_B \,d\mu\\
&=\int_M \left\langle (B^*)^{-1}df, (-L_B^\sslash+M_B)^{-1}\left((B^*)^{-1}df\right)\right\rangle_B \,d\mu\\
\end{align*}

Using Lemma \ref{P6.4} to $C=-L_B^\sslash$ and $D=M_B$, we have:
\begin{align*}
\Var_\mu(f)
&\leq\int_M\left\langle (B^*)^{-1}df, M_B^{-1}\left((B^*)^{-1}df\right)\right\rangle_B \,d\mu\\
&\leq\int_M\left\langle df, \left(B^* M_B (B^*)^{-1}\right)^{-1}df\right\rangle \, d\mu
\end{align*}

Now, when the operator $M_B$ is not uniformly bounded from below by a positive constant, we need to regularize. For all $\varepsilon>0$, the operator $\varepsilon \id -L^{W,B}$ is invertible and we have the following integral representation for all $1$-form $\alpha$:
\begin{equation}
(\varepsilon \id - L^{W,B})^{-1}\alpha = \int_0^{+\infty}{e^{-\varepsilon t} \BQ_t^B\alpha\, dt}.
\label{E6.9}
\end{equation}
Similarly, $(\varepsilon \id-L)$ is invertible on the sub-space of centred functions and we have the integral representation for all centred $f\in\CC^\infty_c(M)$:
\begin{equation}
(\varepsilon \id - L)^{-1}f = \int_0^{+\infty}{e^{-\varepsilon t} \BP_tf\, dt} := g_\varepsilon.
\label{E6.10}
\end{equation}
We have:
\begin{align*}
\Var_\mu(f) 
&= \int_M{f^2\, d\mu}\\
&= \int_M{ f (\varepsilon \id-L)g_\varepsilon\, d\mu}\\
&= \varepsilon\int_M{fg_\varepsilon\, d\mu} + \int_M{f(-L)\left(\int_0^{+\infty}{e^{-\varepsilon t}\BP_tf\, dt}\right)\, d\mu}\\
&= \varepsilon\int_M{fg_\varepsilon\, d\mu} + \int_0^{+\infty}{e^{-\varepsilon t}\int_M{f(-L)\BP_tf\, d\mu}\, dt}\\
&= \varepsilon\int_M{fg_\varepsilon\, d\mu} + \int_0^{+\infty}{e^{-\varepsilon t}\int_M{\left\langle (B^*)^{-1}df,\BQ^B_t((B^*)^{-1} df)\right\rangle_B\, d\mu}\, dt}\\
&= \varepsilon\int_M{fg_\varepsilon\, d\mu} + \int_M{\left\langle (B^*)^{-1}df,\int_0^{+\infty}{e^{-\varepsilon t}\BQ^B_t((B^*)^{-1} df)\, dt}\right\rangle_B\, d\mu}\\
&= \varepsilon\int_M{fg_\varepsilon\, d\mu} + \int_M{\left\langle (B^*)^{-1}df,(\varepsilon \id-L^{W,B})^{-1}((B^*)^{-1} df)\right\rangle_B\, d\mu}.\\
\end{align*}

We can apply Lemma \ref{P6.4} to $\varepsilon \id-L^{W,B} = \varepsilon \id-L^\sslash_B + M_B$. We have:
\begin{equation*}
\Var_\mu(f) \leq \varepsilon\|f\|_{L^2(\mu)}\|g_\varepsilon\|_{L^2(\mu)} + \int_M{\left\langle (B^*)^{-1}df,(M_B)^{-1}((B^*)^{-1} df)\right\rangle_B\, d\mu}.
\end{equation*}
Finally, we have
\begin{equation*}
\varepsilon\|g_\varepsilon\|_{L^2(\mu)} = \left\|\int_0^{+\infty}e^{-t}\BP_{t/\varepsilon}f\, dt\right\|_{L^2(\mu)}\leq \int_0^{+\infty}\int_M e^{-t}(\BP_{t/\varepsilon}f)^2\, d\mu\, dt .
\end{equation*}
By ergodicity of $\BP$ and dominated convergence, this term converges to $0$ as $\varepsilon \to 0$. This ends the proof.
\end{proof}

An immediate corollary of this theorem is the Poincaré inequality.
\begin{theo}[Poincaré inequality - symmetric case]\label{P6.5}
Assuming that $\BB=0$ and that $\rho_B$ is positive, for all $f\in\CC_c^\infty(M)$, we have
\begin{equation*}
\Var_\mu(f)\leq \frac{1}{\rho_B} \int_M{| df|^2\, d\mu},
\end{equation*}
\end{theo}

In the case where $M_B$ is only positive and not uniformly bounded from below (i.e $\rho_B=0$), this inequality is trivially true. Let us give an alternative proof which does not use the generalized Brascamp-Lieb inequality, and thus, avoids the inversion of $L^{W,B}$ and its integral representation.

\begin{proof}
Using a time change and the symmetry of the semigroup $\BQ^B$, we have a new representation of the variance:
\begin{align*}
\Var_\mu(f)
&=2\int_0^\infty\left(\int_M \left\langle (B^*)^{-1} df,  \BQ_{2t}^{B}((B^*)^{-1} df)\right\rangle_B \,d\mu\right)\,dt\\
&=2\int_0^\infty\left(\int_M \left|\BQ_{t}^{B}((B^*)^{-1} df)\right|^2_B \,d\mu\right)\,dt.
\end{align*}

Let 
\begin{equation}
\label{E6.11}
\phi(t)= \int_M \left|\BQ_{t}^{B}((B^*)^{-1} df)\right|^2_B \,d\mu.
\end{equation}

We have 
\begin{align*}
\phi'(t)
&=2\int_M\left\langle \BQ_{t}^{B}((B^*)^{-1} df), L^{W,B} \BQ_{t}^{B}((B^*)^{-1} df) \right\rangle_B\,d\mu\\
&=2\int_M\left\langle \BQ_{t}^{B}((B^*)^{-1} df), (L^\sslash_B) \BQ_{t}^{B}((B^*)^{-1} df) \right\rangle_B\,d\mu\\
&\ \ -2\int_M\left\langle \BQ_{t}^{B}((B^*)^{-1} df), M_B \BQ_{t}^{B}((B^*)^{-1} df) \right\rangle_B\,d\mu\\
&=-2\int_M\left|\nabla\BQ_{t}^{B}((B^*)^{-1} df)\right|^2_B\,d\mu\\
&\ \  -2\int_M\left\langle \BQ_{t}^{B}((B^*)^{-1} df), M_B \BQ_{t}^{B}((B^*)^{-1} df) \right\rangle_B\,d\mu\\
&\leq -2\int_M\left\langle \BQ_{t}^{B}((B^*)^{-1} df), M_B \BQ_{t}^{B}((B^*)^{-1} df) \right\rangle_B\,d\mu\\
&\leq -2  \rho_B\phi(t)
\end{align*}

By Grönwall lemma, this implies
\begin{equation}
\label{E6.12}
\phi(t) \le e^{-2\rho_Bt} \phi(0)=e^{-2\rho_Bt}\int_M|df|^2\,d\mu.
\end{equation}

Integrating on $\R_+$ ends the proof.
\end{proof}

We finish with an interpretation of the Poincaré inequality in terms of spectral gap. 

\begin{prop}\label{P6.6}
Assume that $\BB=0$ and that $\rho_B$ is positive then the spectral gap satisfies 
\begin{equation}
\lambda_1(-L,\mu)\geq\rho_B
\end{equation}
\end{prop}

This is a generalization to Riemannian manifolds of the Chen and Wang formula established in the one dimensional case in \cite{CW}. This spectral gap gives an exponential rate of convergence to equilibrium to the ergodic semigroup $\BP$.
%%%%%%%%%%%%%%%%%%%%%%%%%%%%%%%%%%%%%%%%%%%%%%%%%%%%%%%%%%%%%%%%%%%%%%%%%%%%
\section{Intertwining: general case}\label{Section7}
\setcounter{equation}0
%%%%%%%%%%%%%%%%%%%%%%%%%%%%%%%%%%%%%%%%%%%%%%%%%%%%%%%%%%%%%%%%%%%%%%%%%%%%
The goal of this section is to extend the results of Section \ref{Section6} without the strong condition of Proposition \ref{P5.1}. These results are more theoretical because we will not apply it to any example but they show that our twisting method is strong to perturbations and could be applied to a more general class of twist than the class of Proposition \ref{P5.2}. Actually, we can release all assumptions on the second order operator if we are ready to strengthen the conditions on the potential $M_B$. In this case, the eigenvalue $\rho_B$ is not a good criterion anymore. We need to find a quantity which offsets the lack of symmetry. For all $1$-form $\alpha$, according to \eqref{E5.5}, we have:
\begin{align*}
\int_M{\langle L^{\sslash}_B\alpha,\alpha\rangle_B\, d\mu} 
&= -\int_M{\left| B^*\nabla \alpha\right|^2\,d\mu} +\int_M{\left\langle B^*\nabla \alpha,\BB (B^*\alpha)\right\rangle\, d\mu}\\
&= -\int_M{\left| B^*\nabla \alpha-\frac{1}{2}\BB B^*\alpha\right|^2\,d\mu} +\int_M{\frac{1}{4}\left| \BB B^*\alpha\right|^2\, d\mu}\\
&\leq \int_M\left\langle B^* \alpha, N_B B^* \alpha\right\rangle\, d\mu
\end{align*}

with $\BB$ defined in \eqref{E5.6} and $N_B$ the section of $\End(T^*M)$ defined by
\begin{equation}
N_B(x) = \frac{1}{4} \BB(x)^t\cdot\BB(x) \in \End(T^*_xM). 
\label{E7.1}
\end{equation}

Hence, we have the following lower bound:
\begin{equation}
\int_M{\left\langle \left(-L^{W,B}\right)\alpha,\alpha\right\rangle_B\, d\mu}\geq \int_M{\left\langle B^* \alpha, \left[\left(B^* M_B(B^*)^{-1})\right)^{s} -N_B\right] B^*\alpha\right\rangle\, d\mu}.
\label{E7.3}
\end{equation}

where $\left(B^* M_B(B^*)^{-1})\right)^{s}$ is the symmetric part of $B^* M_B(B^*)^{-1}$ with respect to the Riemannian metric.
So the quantity we need to control seems to be the following:
\begin{equation}
\tilde{\rho}_B= \inf_{x\in M}\left\{\hbox{smallest eigenvalue of }\left(B^* M_B(B^*)^{-1})\right)^{s} -N_B \right\}.
\label{E7.4}
\end{equation}

First, as in the symmetric case, we show the intertwining relation.

\begin{theo}\label{P7.1}
Assume that $\left(B^* M_B(B^*)^{-1})\right)^{s} -(1+\varepsilon) N_B$ is bounded from below for some $\varepsilon>0$. Then the semigroups $\BP$ and $\BQ^B$ are intertwined by~$(B^*)^{-1} d$, i.e for every $f\in\CC^\infty_c(M)$ and $t\geq0$ we have :
\begin{equation*}
(B^{-1})^* d\BP_tf = \BQ_t^B\left((B^{-1})^* df\right).
\end{equation*}
\end{theo}

\begin{proof}
The core of the proof is still the uniqueness of the solution of the same $L^2$ Cauchy problem. We assume that  ${\left(B^* M_B(B^*)^{-1})\right)^{s} -(1+\varepsilon) N_B}$ is non-negative without any loss of generality. Let $F$ be a solution with the zero initial condition. For $\phi\in\CC_c^\infty$ and $\tau>0$, as in the proof of proposition \ref{P6.1}, we have
\begin{equation}
\int_0^\tau\int_M{\phi^2\langle F, (L^\sslash_B -(1+\varepsilon)(B^*)^{-1}N_BB^*) F\rangle_B\, d\mu\, dt}\geq \int_M{\phi^2\frac{1}{2}|F(\cdot,\tau)|^2_B\, d\mu}.
\end{equation}

On the other hand, according to the formula \eqref{E5.5}, we have
\begin{align*}
\int_M{\phi^2\langle F, L^\sslash_B F\rangle_B\, d\mu}
&= -\int_M{\langle \nabla(\phi^2 F), \nabla F\rangle_B\, d\mu} + \int_M{\langle B^*\nabla F, \BB(B^*\phi^2 F)\rangle\, d\mu}\\
&= -\int_M{\phi^2 |\nabla F|^2_B\, d\mu} + \int_M{\phi^2\langle B^*\nabla F, \BB(B^*\phi^2 F)\rangle\, d\mu}\\
&\qquad -2\int_M{\langle \nabla\phi\otimes F,\phi\nabla F\rangle_B\, d\mu}\\
&= -\int_M{\phi^2 |B^*\nabla F-\frac{1}{2}\BB B^*F|^2\, d\mu} + \int_M{\phi^2\langle  F, N_B F\rangle\, d\mu}\\
&\qquad -2\int_M{\langle \nabla\phi\otimes F,\phi(\nabla F-\frac{1}{2}(B^*)^{-1}\BB B^*F)\rangle_B\, d\mu}\\
&\qquad -2\int_M{\langle \nabla\phi\otimes F,\phi\frac{1}{2}(B^*)^{-1}\BB B^*F\rangle_B\, d\mu} .\\
\end{align*}

According to Cauchy-Schwarz inequality, for every $\lambda, k>0$, we have:
\begin{align*}
2|\langle\nabla\phi\otimes F, \phi(\nabla F-\frac{1}{2}(B^*)^{-1}\BB B^*F)\rangle_B| &\leq \lambda|\nabla\phi\otimes F|^2_B +\frac{1}{\lambda}\phi^2|B^*\nabla F-\frac{1}{2}\BB B^*F|^2\\
2|\langle\nabla\phi\otimes F, \phi\frac{1}{2}(B^*)^{-1}\BB B^*F\rangle_B| &\leq k|\nabla\phi\otimes F|^2_B +\frac{1}{k}\phi^2|\frac{1}{2}\BB B^*F|^2\\
\end{align*}

So, we have:
\begin{align*}
\int_M{\phi^2\langle F, L^\sslash_B F\rangle_B\, d\mu} \leq &\left(\frac{1}{\lambda}-1\right)\int_M{\phi^2 |B^*\nabla F-\frac{1}{2}\BB B^*F|^2\, d\mu}\\
 &+\left(1+\frac{1}{k}\right) \int_M{\phi^2\langle F, N_B F\rangle\, d\mu} + (\lambda+k)\int_M{|\nabla\phi|^2|F|^2_B\, d\mu} .\\
\end{align*}
Combining the above inequalities, we obtain that there exists a $c>0$ such that for every $\phi\in\CC_c^\infty(M)$, and every $\tau>0$
\begin{equation}
\frac{1}{2}\int_M{\phi^2|F(\cdot,\tau)|^2_B\, d\mu} \leq c\int_0^\tau\int_M{|\nabla\phi|^2|F|^2_B\, d\mu\, dt}.
\label{E7.14}
\end{equation}

Using a sequence of cut-off functions, we prove that $F=0$. The end of the proof follows the proof of Theorem \ref{P6.2} without any differences.
\end{proof}

Remark that under the condition of proposition \ref{P7.1}, $\tilde{\rho}_B$ is bounded from below, since $N_B$ is non-negative. But unlike in Theorem \ref{P6.1}, this proof requires a stronger condition. 

Back to our assumptions of ergodicity and probability measure, the intertwining relation of Proposition \ref{P7.1} implies the covariance's representation \eqref{E6.6}. This brings Brascamp-Lieb and Poincaré type inequalities.

\begin{theo}[Poincaré inequality]\label{P7.2}
Assume that ~${\left(B^* M_B(B^*)^{-1})\right)^{s} -(1+\varepsilon) N_B}$ is bounded from below for some $\varepsilon>0$ and that $\tilde{\rho}_B$ is positive. Then for every ~${f\in\CC^\infty_c(M)}$, we have :
\begin{equation*}
\Var_\mu(f)\leq \frac{1}{\tilde{\rho}_B} \int_M{| df|^2\, d\mu}.
\end{equation*}
\end{theo}

\begin{proof}
Let $f\in\CC^\infty_c(M)$ and $F_t = \BQ^B_t\left((B^*)^{-1}df\right)$. As in \eqref{E6.8}, we set 
\begin{equation}
\label{E7.6}
\phi(t)= \int_M \left|F_t\right|^2_B \,d\mu
\end{equation}
and we have the following representation of the variance 
\begin{equation}
\Var_\mu(f) = \int_0^{+\infty}\phi(t)\, dt.
\label{E7.7}
\end{equation}

We have:
\begin{align*}
\phi'(t) 
&= 2\int_M{\left\langle F_t, L^{W,B}F_t\right\rangle_B\, dt}\\
&\leq -2\int_M{\left\langle B^* F_t, \left[\left(B^* M_B(B^*)^{-1})\right)^{s} -N_B\right] B^*F_t\right\rangle\, d\mu}\\
&\leq -2\tilde{\rho}_B\phi(t)\\
\end{align*}

So we have
\begin{equation}
\phi(t)\leq e^{-2\tilde{\rho}_Bt}\int_M{| df|^2\, d\mu}.
\label{E7.8}
\end{equation}
Integrating on $\R_+$ gives the results.
\end{proof}

As for the Theorem \ref{P6.3}, the result still make sense when $\tilde{\rho}_B=0$. With the same kind of hypothesis, we can also prove a generalized Brascamp-Lieb inequality.

\begin{theo}[Generalized Brascamp-Lieb inequality]\label{P7.3}
Assume that the operator~$\left(B^* M_B(B^*)^{-1})\right)^{s} -(1+\varepsilon) N_B$ is bounded from below for some~${\varepsilon>0}$ and is positive definite for $\varepsilon=0$, then for every $f\in\CC^\infty_c(M)$ we have :
\begin{equation}
\Var_\mu(f)\leq \int_M{\left\langle df,\left[\left(B^* M_B(B^*)^{-1})\right)^{s} -N_B\right]^{-1} df\right\rangle\, d\mu}.
\end{equation}
\end{theo}

\begin{proof}
First, let assume that $\tilde{\rho_B}$ is positive. Equation \eqref{E7.3} implies that for all~$1$-form~$\alpha$ we have:
\begin{equation}
\int_M{\langle (-L^{W,B})\alpha,\alpha\rangle_B\, d\mu}\geq \tilde{\rho_B}\int_M{|\alpha|^2_B\, d\mu}.
\label{E7.10}
\end{equation}

As in the proof of Theorem \ref{P6.3}, $-L^{W,B}$ is invertible with the same integral representation. So 
\begin{equation}
\Var_\mu(f)=\int_M \left\langle (B^*)^{-1}df, (-L^{W,B})^{-1}\left((B^*)^{-1}df\right)\right\rangle_B \,d\mu .
\label{E7.11}
\end{equation}

Furthermore, we have:
\begin{align*}
\int_M{\langle \alpha,(-L^{W,B})\alpha\rangle_B\, d\mu} 
&\geq \int_M\left\langle B^*\alpha, \left[\left(B^*M_B(B^*)^{-1}\right)^s-N_B\right] B^*\alpha\right\rangle\, d\mu\\
&\geq \int_M\left\langle \alpha, (B^*)^{-1}\left[\left(B^*M_B(B^*)^{-1}\right)^s-N_B\right]B^* \alpha\right\rangle_B\, d\mu.\\
\end{align*}

As $(B^*)^{-1}\left[\left(B^*M_B(B^*)^{-1}\right)^s-N_B\right]B^*$ is symmetric with respect to $\langle\cdot,\cdot\rangle_B$ and positive by assumption, we can use Lemma \ref{P6.1} to obtain
\begin{equation*}
\Var_\mu(f) \leq \int_M \left\langle (B^*)^{-1}df, (B^*)^{-1}\left[\left(B^*M_B(B^*)^{-1}\right)^s-N_B\right]^{-1}B^*\left((B^*)^{-1}df\right)\right\rangle_B \,d\mu .
\end{equation*}

Now, if $\tilde{\rho}_B=0$, we regularize as in the proof of theorem \ref{P6.3}. It ends the proof.
\end{proof}

We also obtain a bound for the spectral gap.

\begin{prop}\label{P7.4}
Assume that $\left(B^* M_B(B^*)^{-1})\right)^{s} -(1+\varepsilon) N_B$ is bounded from below for some $\varepsilon>0$ and that $\tilde{\rho}_B$ is positive. Then the spectral gap $\lambda_1(-L,\mu)$ satisfies:
\begin{equation}
\lambda_1(-L,\mu)\geq \tilde{\rho}_B.
\end{equation}
\end{prop}

Remark that if the hypothesis of Proposition \ref{P5.1} are satisfied, then $N_B=0$ and~$B^* M_B(B^*)^{-1}$ is symmetric and so $\rho_B= \tilde{\rho}_B$. In particular, Theorem \ref{P7.2} (respectively \ref{P7.3} and \ref{P7.4}) can be applied to small perturbations of generators satisfying the conditions of Theorem \ref{P6.5} (or \ref{P6.3} and \ref{P6.6}) and the bounds obtained are stable with respect to perturbations.

%%%%%%%%%%%%%%%%%%%%%%%%%%%%%%%%%%%%%%%%%%%%%%%%%%%%%%%%%%%%%%%%%%%%%%%%%%%%
\section{Examples in radially symmetric surfaces}\label{Section8}
\setcounter{equation}0
%%%%%%%%%%%%%%%%%%%%%%%%%%%%%%%%%%%%%%%%%%%%%%%%%%%%%%%%%%%%%%%%%%%%%%%%%%%%
In this section, we illustrate our results with three examples. Each one corresponds to a different case where Bakry-\'Emery criterion is not satisfied : strictly convex in each point but $\rho=0$, strictly concave in a compact region and a not upper bounded $\rho$, and strictly concave in a compact region and $\rho$ upper bounded. We also give heuristic ideas to find adequate twists. We even improve a lower bound in a classical example. The difficulty is to find a concession between interesting examples (manifold and measure) and calculability. In our examples, the measure~$\mu$ will be classical but the manifold will be from casual (as hyperbolic plan) to quite exotic. Our manifold $M$ will be a two dimensional radially symmetric space with global polar chart ~${(r,\theta)\in \R_+\times \S^1}$, such that, in this coordinates, the metric is given by : 
\begin{equation}\label{E8.2}
ds^2 = dr^2 + f(r)^2d\theta^2
\end{equation}
where $f : \R_+\to\R_+$ is a smooth function satisfying $f(r)=0$ if and only if~${r=0}$ and~$f'(0)=1$. This model includes the surfaces of constant curvature : hyperbolic plan  $\H^2$ with $f(r) = \sinh(r)$ or Euclidean plan $\R^2$ with $f(r) = r$. It also includes surface of revolution. The Riemanian volume measure of such a manifold is : $\vol(drd\theta) = f(r)drd\theta$. For every smooth function ~$\phi$, in the orthonormal basis~$(\partial_r, \frac{1}{f(r)}\partial_\theta)$, we have :
\begin{equation}
\nabla\phi  =\left(\begin{aligned}&\partial_r\phi\\ &\frac{1}{f}\partial_\theta\phi\\\end{aligned}\right),
\end{equation}
\begin{equation}
\nabla^2\phi  =\left(\begin{array}{cr}\partial^2_r\phi & \frac{1}{f}\partial^2_{r,\theta}\phi -\frac{f'}{f^2}\partial_\theta\phi \\\frac{1}{f}\partial^2_{r,\theta}\phi -\frac{f'}{f^2}\partial_\theta\phi &\frac{1}{f^2}\partial^2_\theta\phi + \frac{f'}{f}\partial_r\phi\\\end{array}\right),
\end{equation}
\begin{equation}
\Ric = -\frac{f''}{f}\id.
\end{equation}
For more details on radially symmetric manifold see \cite{Hsu} or \cite{Rou}. In all our examples, twists have the form of Proposition \ref{P5.2} : ${B(x) = b(x) \id_{T_xM}}$ with $b = \exp(U)$ a radial positive function. With this special form of twist, we have :
\begin{equation}
(B^*)^{-1}L^\sslash(B^*) = b^{-1}L(b) \id_{T_xM},
\end{equation}
and
\begin{equation}\label{E8.1}
b^{-1}L(b) = \Delta U -\langle \nabla V, \nabla U\rangle +|\nabla U|^2. 
\end{equation}
As in the Euclidean case, a usual choice of twist is $U=\varepsilon V$ but we will also see some cases where it is not enough to obtain a positive $\rho_B$. As metrics, measures and twists are radial, we will only use the variable $r$ with a slight abuse of notation.

Our first example is the case of generalized Cauchy measures on $\R^2$. It have been studied in \cite{Ngu} for weighted Poincaré inequalities and in \cite{BJM} for bounds on the spectral gap, both in any dimension $n\geq2$. We show that our method can improve the previous lower bounds for $n=2$. This example also illustrate how using Riemannian geometry can help for measures in an Euclidean space. On $\R^2$, we define the function $\sigma^2 (x)= 1+|x|^2$. For $\beta>1$, we define the differential operator ~$L_\beta$ by :
\begin{equation}
L_\beta f(x) = \sigma^2(x)\Delta_E f(x) -2(\beta-1)x.\nabla_Ef(x),\, \forall f\in\CC^\infty(\R^2),\, \forall x\in\R^2,
\end{equation}
where $\Delta_E$ and $\nabla_E$ stand for the Euclidean Laplacian and gradient. This operator admits a unique invariant probability $\mu_\beta$ whose density with respect to the Lebesgue measure is proportional to $(\sigma^2)^{-\beta}$. Remark that for $\beta\leq1$, it does not define a finite measure any more. The form of the generator $L_\beta$ suggests that the Euclidean geometry is not adapted to the problem. The relevant space is the manifold $M$ which have a global Cartesian chart $\R^2$ and whose metric is given by 
$$ds^2 = \frac{dx_1^2 +dx_2^2}{\sigma^2}.$$
In order to obtain an expression of the metric as in \eqref{E8.2}, we use the appropriate generalized polar coordinates :
\begin{equation}
(x_1,x_2) = \left(\sinh(r)\cos(\theta), \sinh(r)\sin(\theta)\right),\, (r,\theta)\in \R_+^*\times \S^1.
\end{equation}
In the chart $(r,\theta)$, the metric has the desired form, with $f = \tanh$. Then, we need to find the density of $\mu_\beta$ with respect to the Riemannian volume. We have :
\begin{align*}
d\mu_\beta 
&= Z(\sigma^2)^{-\beta}dxdy\\
&= Z\cosh^{-2\beta} \sinh\cosh dr d\theta\\
&= Z\cosh^{-2(\beta-1)}\tanh dr d\theta\\
&= Z\cosh^{-2(\beta-1)}\vol(drd\theta),\\
\end{align*}
with $Z$ the normalization constant. Then, if we set $V(r) = 2(\beta-1)\ln(\cosh(r))$ for $(r,\theta)\in \R_+$, the generator $L_\beta$ has the expression \eqref{E2.1} and we can apply our method. Firstly, the operator $\MM$ is an homothetic transformation :
\begin{equation}
\MM = \frac{2\beta}{\sigma^2} \id.
\end{equation} 
For each $x\in M$, $\MM$ is strictly convex but globally, $\MM$ is only convex : $\rho =0$. We try a twist of the shape $\exp(\varepsilon V)$. Using the formula \eqref{E8.1}, for all $r\geq0$, we have~:
\begin{equation}
\rho_B(r) = 2\beta -4\varepsilon(\beta-1) + [4\varepsilon(1-\varepsilon)(\beta-1)^2-(2\beta -4\varepsilon(\beta-1))]\tanh^2(r).
\end{equation}
The function $\rho_B$ is monotonous and can be bounded from below by the minimum between its value in $r=0$ and its limit as $r\to+\infty$ : 
\begin{equation}
\rho_B \geq \min\left\{ 2\beta -4\varepsilon(\beta-1), 4\varepsilon(1-\varepsilon)\right\}.
\end{equation}
The optimal parameter is :
\begin{equation}
\varepsilon_0=\left\{ \begin{array}{lcc} 
\frac{1}{2} & \text{if} &1<\beta\leq 1+\sqrt{2}\\ 
\frac{\beta+\sqrt{(\beta-1)^2-1}}{2(\beta-1)} & \text{if} &1+\sqrt{2}\leq\beta\\ 
\end{array}\right.
\end{equation}.

\begin{corol}\label{P8.1}
The spectral gap of the operator $L_\beta$ is bounded from below by :
\begin{equation}
\lambda_1(L_\beta) \geq \left\{ \begin{array}{lcc} 
(\beta-1)^2 & \text{if} &1<\beta\leq 1+\sqrt{2}\\ 
2\sqrt{(\beta-1)^2-1} & \text{if} &1+\sqrt{2}\leq\beta\\ 
\end{array}\right. .
\end{equation}
\end{corol}
Back to $\R^n$, this spectral gap is interpreted as weighted Poincaré inequality :
\begin{equation*}
\Var_\mu(f)\leq \frac{1}{\rho_B}\int_M|\nabla_Ef|^2\sigma^2d\mu_\beta,\; \forall f\in\CC_0^\infty(M).
\end{equation*} 
Remark that for $\beta\geq 1+\sqrt{2}$, the optimal $\varepsilon_0$ corresponds to the case where $\rho_B$ is a constant function. The best lower bound known for this spectral gap, in \cite{BJM}, are :
\begin{equation*}
\lambda_1(L_\beta) \geq \left\{ \begin{array}{lcc} 
(\beta-1)^2 & \text{if} &1<\beta\leq \frac{3+\sqrt{5}}{2}\\ 
\beta & \text{if} &\frac{3+\sqrt{5}}{2}\leq\beta\\ 
\end{array}\right. .
\end{equation*} 
So our method improves the result for $\beta\geq\frac{3+\sqrt{5}}{2}$. Actually, \cite{BJM} also gives upper bounds : 
\begin{equation*}
\lambda_1(L_\beta) \leq \left\{ \begin{array}{lcc} 
(\beta-1)^2 & \text{if} &1<\beta\leq \frac{3+\sqrt{5}}{2}\\ 
2(\beta-1) & \text{if} &\frac{3+\sqrt{5}}{2}\leq\beta\\ 
\end{array}\right. ,
\end{equation*}
and for $\beta\geq 3$, it is proved in \cite{Ngu} that $\lambda_1(L_\beta) = 2(\beta-1)$. This shows that our lower bound is optimal for $\beta\leq 1+\sqrt{2}$ and has the good asymptotic for $\beta\to+\infty$, even if our choice of twist, a priori, cannot pretend to be optimal.

Our second example is the case of exponential power measures on the hyperbolic plan. We set $M=\H^2$, $f=\sinh$ and for $\alpha>2$,
\[V(r)= \frac{r^\alpha}{\alpha},\,\forall r\in\R_+.\]
Remark that the measure associated to $V$ is finite for $\alpha>1$. The generator associated to this measure is 
\begin{equation}
L_\alpha = \partial_r^2 + \frac{1}{\tanh(r)}\partial_r +\frac{1}{\sinh^2(r)}\partial^2_\theta-r^{\alpha-1}\partial_r.
\end{equation}
Using a result from \cite{Wang}, these measures satisfy a Log-Sobolev inequality for $\alpha \geq2$. The limit case $\alpha=2$, corresponds to the radial hyperbolic Ornstein-Uhlenbeck process. We will discuss at the end why it must be excluded by our method. The  smallest eigenvalue of the potential $\MM$ is :
\begin{equation}
\rho_\alpha(r) = \min\left\{(\alpha-1)r^{\alpha-2}, \frac{r^{\alpha-1}}{\tanh(r)}\right\} -1,\, \forall r\in\R_+,
\end{equation}
so its infimum is $\rho = -1$. Then, we know that the semigroups ~$\BP$ and ~$\BQ$ are intertwined but as $\rho$ is not positive, we cannot directly use it in terms of functional inequalities. It is a case where twisting is needed. In these cases, the operator $\MM$ is concave in a neighbourhood of the origin and strictly convex outside . We need a choice of $b$ which counter the negativity of $\Ric$ around ~${r=0}$. A direct calculation show that $U=\varepsilon V$ cannot achieve this goal. We propose the following function :
$$U_{\varepsilon, \eta}(r) =  \frac{1}{2}V(r) -\varepsilon\frac{r^2}{2} + \eta \ln(\cosh(r)),\, \forall r\in\R_+,$$
with $\varepsilon$ and $\eta$ parameters which should be fitted. Let us explain this choice. In the expression \eqref{E8.1} there is the beginning of the square $|\nabla U-\nabla V/2|^2$ which explains the term $V/2$. The second term is inspired by the one-dimensional case in  \cite{BJ}. Its Laplacian should help in $r=0$ because it will not vanish there. The third term is directly linked to the metric : it is a primitive of $f/f'$. This makes appear the Ricci curvature in the expression of $b^{-1}L(b)$. With this choice of $U$, for all $r\in\R_+$, we have :
\begin{align*}
b^{-1}L_\alpha(b)(r) 
&=  2\eta-\varepsilon +\frac{\alpha-1}{2}r^{\alpha-2}+\varepsilon^2r^2-\frac{r^{2\alpha-2}}{4} +\frac{r^{\alpha-1}-2\varepsilon r}{2\tanh(r)}\\ 
&\quad -2\varepsilon\eta r\tanh(r) -\eta(1-\eta)\tanh^2(r).\\
\end{align*}
Now, we need to find whether there exists coefficient $\varepsilon$ and $\eta$ such that $\rho_B$ is positive and for which coefficient it is optimal. Unfortunately, it seems difficult to give explicit bounds in all generality. Nevertheless, numerically, we find the following bounds :
$$\begin{array}{|c||c|c|c|c|c|}
\hline
\alpha & 2.01 & 3 & 3.5 & 4 & 4.5\\
\hline
(\varepsilon,\eta) & (0.5, -0.007)  & (1.59, 0.8) & (1.83, 1.15) & (2.45, 1.84) & (2.78, 2.279)\\
\hline
\lambda_1(L_\alpha)\geq\cdot & 6.10^{-4} & 0.47 & 0.34 & 0.21 & 2.10^{-3}\\
\hline
\end{array} $$
Remark that the method developed in \cite{BJM} needs $\alpha >3$ to obtain results. These bounds do not pretend to be optimal. Perhaps another kind of $b$ could have brought better results, especially for $\alpha = 4.5$, or $\alpha = 5$ for which we did not find good parameters. It could be interesting to bring some upper-bound on the spectral gap, by a testing on examples or by other means, to discuss the relevancy of our lower-bound. Concerning the decay of our bounds for $\alpha$ near $2$, it was expected. Indeed, for $\alpha = 2$, $\rho_2$ is a constant function, equal to zero. But as explained in \cite{ABJ}, an integration by parts shows that 
\begin{equation}
\int_M-b^{-1}L(b)\, d\mu = -\int_M|db|^2\, d\mu.
\end{equation}
If $b$ is not constant, it will be negative in a region and so $\rho_B$. Twisting with a function absolutely needs a region where $\MM$ is strictly convex. It is not the case of the hyperbolic Ornstein-Uhlenbeck generator. It could be interesting to look at more complex $B$, in the way of Section \ref{Section7}.

In our last, example, we have a similar situation : bounded region of strict concavity and strict convexity elsewhere but with another constraint : $\rho$ is upper bounded. Let $M$ be the revolution surface generated by the rotation around the ordinates axis of the curve
\begin{equation}
y=\frac{1}{\sqrt{1+x^2}},\, x\in\R_+. 
\end{equation}

It is a regularized version of the surface generated by $y=1/x$. In an adapted polar chart, its metric has the form \eqref{E8.2} but we don't have any explicit formula for $f$. However, we can find sufficiently sharp estimates of it. For a surface in $\R^3$ parametrized as 
\begin{equation}
S=\left\{ (f(r) \cos(\theta), f(r)\sin(\theta), g(r) ):(r,\theta)\in\R_+\times \S^1 \right\},
\end{equation}
the metric has the form : $ds^2 = (f'^2+g'^2)dr^2 +f^2d\theta^2$. Using the relation between~$f$ and~$g$ given by the generating curve, we can prove that $f$ satisfies the equation :
$$f' = \frac{1}{\sqrt{1+\frac{f^2}{(1+f^2)^3}}}.$$
We obtain the following bounds : for all $r\in\R_+$,
\begin{align*}
\alpha r &\leq f(r) \leq r\\
\frac{1}{\sqrt{1+\frac{r^2}{(1+\alpha^2 r^2)^3}}} &\leq f'(r) \leq \frac{1}{\sqrt{1+\frac{\alpha^2r^2}{(1+ r^2)^3}}}\\,
\end{align*}
with $\alpha = \sqrt{\frac{27}{31}}$. The Ricci curvature has an explicit formula in function of $f$ : 
\begin{equation}
\Ric = \frac{(1-2f^2)(1+f^2)^2}{(f^2+(1+f^2)^3)^2}.
\end{equation}
This formula gives us lower and upper bounds on $\Ric$. In particular, we know that~$\Ric(0)=1$, then it decreases to a negative minimum (which is numerically in the range ~${-0.050<\min\Ric <-0.049}$) and then it increases and vanishes. We are interested in radial Gaussian measures on this manifold : for $\gamma>0$,
$$V_\gamma(r) = \frac{\gamma r^2}{2},\, \forall r\in\R_+.$$
The smallest eigenvalue of $\MM$ is :
$$\rho (r) = \gamma\min\left\{1, r\frac{f'(r)}{f(r)} \right\} + \Ric(r),\, \forall r\geq0.$$
If $\alpha\gamma$ is bigger to $-\min\Ric$, $\rho$ will be positive. So we are mainly interested in the case of small $\gamma$ such that twistings are unavoidable, but also in cases of "big" $\gamma$ as we shall see. Thanks to the Ricci curvature, we know that $\rho(r)$ is positive for small $r$ and tends to $\gamma>0$ as $r$ tends to $+\infty$. We need to compensate a compact region of negativity. Here, we use the radial function $U_{\varepsilon,\omega,k}$ defined by
\begin{equation}
U_{\varepsilon,\omega,k}(r) = \int_0^r \varepsilon \sin(\omega t) e^{-k t}\, dt,\, \forall r\in\R_+,
\end{equation} 
where $\varepsilon$, $\omega$ and $k$ are parameters. This goal of this quite unusual function is to give to $b^{-1}L(b)$ the shape of $\Ric$. The decreasing exponential term is explained by the vanishing of $\Ric$, it is linked to the boundedness of $\rho$. The goal of the sinusoidal term is to create a peak, compensating Ricci's minimum. The exponential has to damp the following peaks. According to equation \eqref{E8.1}, for all $0\leq r$, we have :
\begin{equation}
b^{-1}L(b) (r) = \varepsilon\left[\omega\cos(\omega r) + \sin(\omega r)\left(-k + \frac{f'(r)}{f(r)} -\gamma r + \varepsilon \sin(\omega r) e^{-k r}\right)\right]e^{-k r}.
\end{equation}

Again, we can obtain numerical lower bound of the spectral gap. With the parameters $\varepsilon =0.217$, $\omega = 2.022$, $k = 1.7$, we have :

$$\begin{array}{|c||c|c|c|c|c|}
\hline
\gamma &0.01 & 0.02 & 0.03 & 0.04 & 0.06 \\
\hline
\lambda_1(L_\gamma)\geq\cdot & 7.10^{-3} & 0.01 & 0.02 & 0.03 & 0.05 \\
\hline
\end{array} $$
As expected, this bound are smaller than $\gamma$, the Euclidean bound, although they do not seem very sharp, in particular, for very small $\gamma$. It could be linked to the choice of twist $b$. When $\gamma$ become smaller, the region where $\MM$ is strictly concavity become larger. This explains why our choice of twist is less adapted for small $\gamma$. To finish, this example shows that twisting is not the last resort method of spectral gap research and can also be interesting even if the Backy-\'Emery criterion is satisfied. Looking at $\alpha\gamma$ slightly bigger than $-\min\Ric$, the upper bound $\rho \leq \gamma+\min\Ric$ of Bakry-\'Emery is very small but as shown in the array above, we still obtain reasonable bounds ($\gamma = 0.06$). For $\gamma =1$, where we have more room for our parameters, we can obtain a lower bound $\lambda_1(L)\geq 0.98$ (here, we use $k = 2.5$) while $\rho$ is in the range $0.95<\rho<0.951$.
\bibliographystyle{plain}
\bibliography{biblio}

\begin{thebibliography}{10}

\bibitem{ABJ}
Marc Arnaudon, Michel Bonnefont, and Ald\'eric Joulin.
\newblock Intertwinings and generalized brascamp-lieb inequalities.
\newblock {\em Rev. Mat. Iberoam.}, 34:1021--1054, 2018.

\bibitem{ACT}
Marc Arnaudon, Kol\'eh\`e~Abdoulaye Coulibaly, and Anton Thalmaier.
\newblock Horizontal diffusion in {$C^1$} path space.
\newblock In {\em S{\'e}minaire de Probabilit{\'e}s XLIII}, volume 2006 of {\em
  Lecture Notes in Mathematics}, pages 73--94. Springer, 2011.

\bibitem{Bak}
Dominique Bakry.
\newblock Un crit\`ere de non explosion pour certaine diffusions sur une
  vari\'et\'e riemannienne compl\`ete.
\newblock {\em C.R. Acad. Sci. Paris}, S\'er. I Math. 303:23--26, 1986.

\bibitem{BCG}
Dominique Bakry, Patrick Cattiaux, and Arnaud Guillin.
\newblock Rate of convergence for ergodic continuous markov processes: Lyapunov
  versus poincaré.
\newblock {\em J. Funct. Anal.}, 254(3):727 -- 759, 2008.

\bibitem{BE}
Dominique Bakry and Michel \'Emery.
\newblock Diffusions hypercontractives.
\newblock In {\em S\'eminaire de probabilit\'es de Strasbourg XIX}, volume 1123
  of {\em Lecture Notes in Mathematics}, pages 177--206. Springer, 1985.

\bibitem{BJL}
Dominique Bakry, Ivan Gentil, and Michel Ledoux.
\newblock {\em Analysis and Geometry of Markov Diffusion operators}, volume 348
  of {\em Grundlehren der mathematischen Wissenschaften}.
\newblock Springer, 2014.

\bibitem{BJ}
Michel Bonnefont and Ald\'eric Joulin.
\newblock Intertwining relations for one-dimensional diffusions and application
  to functional inequalities.
\newblock {\em Pot. Anal.}, 41(4):1005--1031, 2014.

\bibitem{BJM}
Michel Bonnefont, Aldéric Joulin, and Yutao Ma.
\newblock Spectral gap for spherically symmetric log-concave probability
  measures, and beyond.
\newblock {\em J. Funct. Anal.}, 270(7):2456 -- 2482, 2016.

\bibitem{BL}
Herm~Jan {Brascamp} and Elliott~H. {Lieb}.
\newblock {On extensions of the Brunn-Minkowski and Prekopa-Leindler theorems,
  including inequalities for log concave functions, and with an application to
  the diffusion equation.}
\newblock {\em J. Funct. Anal.}, 22:366--389, 1976.

\bibitem{CC-EL}
Eric~A. Carlen, Dario Cordero-Erausquin, and Elliott~H. Lieb.
\newblock Asymmetric covariance estimates of brascamp–lieb type and related
  inequalities for log-concave measures.
\newblock {\em Ann. Inst. H. Poincaré Probab. Statist.}, 49(1):1--12, 02 2013.

\bibitem{CGWW}
Patrick Cattiaux, Arnaud Guillin, Feng-Yu Wang, and Liming Wu.
\newblock Lyapunov conditions for super poincaré inequalities.
\newblock {\em J. Funct. Anal.}, 256(6):1821 -- 1841, 2009.

\bibitem{ChJ}
Djalil Chafa{\"i} and Ald\'eric Joulin.
\newblock Intertwining and commutation relations for birth-death processes.
\newblock {\em Bernoulli}, 19(5A):1855--1879, 2013.

\bibitem{CW}
Mu-Fa Chen and Feng yu~Wang.
\newblock Estimation of spectral gap for elliptic operators.
\newblock {\em Trans. Amer. Math. Soc.}, 349(3):1239--1267, 1997.

\bibitem{ELL}
K.~David Elworthy, Yves Le~Jan, and Xue-Mei Li.
\newblock {\em On the Geometry of Diffusion Operators and Stochastic Flows},
  volume 1720 of {\em Lecture Notes in Mathematics}.
\newblock Springer, 1999.

\bibitem{ELL1}
K.~David Elworthy, Yves Le~Jan, and Xue-Mei Li.
\newblock {\em The geometry of filtering}.
\newblock Birkhäseur, 2010.

\bibitem{EL2}
K.~David Elworthy and Xue-Mei Li.
\newblock Bisimut type formulae for differential forms.
\newblock {\em Comptes Rendus de l'Académie des Sciences - Series I -
  Mathematics}, 327(1):87 -- 92, 1998.

\bibitem{EL1}
K.D. Elworthy and X.M. Li.
\newblock Formulae for the derivatives of heat semigroups.
\newblock {\em Journal of Functional Analysis}, 125(1):252 -- 286, 1994.

\bibitem{Hel}
Bernard Helffer.
\newblock {\em Semiclassical Analysis, Witten Laplacians, and Statistical
  Mechanics}.
\newblock World Scientific, 2002.

\bibitem{Hsu}
Elton Hsu.
\newblock {\em Stochastic Analysis on Manifolds}, volume~38 of {\em Graduates
  Studies in Mathematics}.
\newblock American Mathematical Society, 2002.

\bibitem{Led2}
Michel Ledoux.
\newblock A heat semigroup approach to concentration on the sphere and on a
  compact riemannian manifold.
\newblock {\em Geom. Funct. Anal.}, 2(2):221--224, 1992.

\bibitem{Led3}
Michel Ledoux.
\newblock Concentration of measure and logarithmic sobolev inequalities.
\newblock In {\em S\'eminaire de probabilit\'es de Strasbourg XXXIII}, volume
  1857 of {\em Lecture Notes in Mathematics}, pages 120--216. Springer, 1999.

\bibitem{Led1}
Michel Ledoux.
\newblock {\em The Concentration of Measure Phenomenon}, volume~89 of {\em
  Mathematical Surveys and Monographs}.
\newblock American Mathematical Society, 2001.

\bibitem{Led}
Michel Ledoux.
\newblock Logarithmic sobolev inequalities for unbounded spin systems
  revisited.
\newblock In {\em S\'eminaire de probabilit\'es de Strasbourg XXXV}, volume
  1755 of {\em Lecture Notes in Mathematics}, pages 167--194. Springer, 2001.

\bibitem{Li}
Peter Li.
\newblock Uniqueness of {$L^1$} solutions for the laplace equation and the heat
  equation on riemannian manifolds.
\newblock {\em J. Differential Geometry}, 20(2):447--457, 1984.

\bibitem{LiXM1}
Xue-Mei Li.
\newblock {\em Stochastic flows on non-compact manifolds}.
\newblock PhD thesis, University of Warwick, 1992.

\bibitem{LiXM2}
Xue-Mei Li.
\newblock Strong p-completeness of stochastic differential equations and the
  existence of smooth flows on noncompact manifolds.
\newblock {\em Probability Theory and Related Fields}, 100:485--511, 1992.

\bibitem{LiXM3}
Xue-Mei Li.
\newblock {On Extensions of Myers' Theorem}.
\newblock {\em Bulletin of the London Mathematical Society}, 27(4):392--396,
  1995.

\bibitem{Mey}
Paul-Andr\'e Meyer.
\newblock G\'eom\'etrie diff\'erentielle stochastique (bis).
\newblock In {\em S\'eminaire de probabilit\'es de Strasbourg XVI}, volume 921
  of {\em Lecture Notes in Mathematics}, pages 165--207. Springer, 1982.

\bibitem{Ngu}
Van~Hoang Nguyen.
\newblock Dimensional variance inequalities of brascamp-lieb type and a local
  approach to dimensional pr\'ekopa's theorem.
\newblock {\em J. Funct. Anal.}, 266(2):931 -- 955, 2014.

\bibitem{RS}
Michael Reed and Barry Simon.
\newblock {\em I: Functional Analysis}.
\newblock Methods of Modern Mathematical Physics. Elsevier Science, 1980.

\bibitem{Rou}
François Rouvi\`ere.
\newblock {\em Initiation à la géométrie de Riemann}, volume 115 of {\em
  Math\'ematiques en devenir}.
\newblock Calvage \& Mounet, 2016.

\bibitem{Stri}
Robert Strichartz.
\newblock Analysis of the laplacian on the complete riemannian manifold.
\newblock {\em J. Funct. Anal.}, 52(1):48 -- 79, 1983.

\bibitem{Wang}
Feng-Yu Wang.
\newblock Log-sobolev inequalities: Different roles of ric and hess.
\newblock {\em Ann. Probab.}, 37(4):1587--1604, 07 2009.

\end{thebibliography}
\end{document}